\documentclass[a4paper,12pt]{article}
\usepackage{amsthm}
\usepackage{extpfeil}

\newlength{\abstractwidth}
\renewcommand{\thanks}[1]{\footnote{#1}} 

\newcommand{\be}{\begin{equation}}
\newcommand{\bea}{\begin{eqnarray}}
\newcommand{\eea}{\end{eqnarray}} \newcommand{\ee}{\end{equation}}
 
 \def\ba{\begin{eqnarray}}
\def\ea{\end{eqnarray}}

\def\o{\omega}

\def\o{\omega}

\def\al{\alpha}
\def\b{\beta}

\def\e{\varepsilon}

\def\o{\omega}

\def\ti{\tilde}

\def\Z{{\bf Z}}

\def\R{{\bf R}}

\def\cO{{\cal O}}

\def\cO{{\cal O}}

\def\[{{\bf [}}
\def\]{{\bf ]}}

\def\pl{\partial}

\begin{document}
\newtheorem{theorem}{Theorem}
\newtheorem{proposition}{Proposition}
\newtheorem{lemma}{Lemma}
\newtheorem{corollary}{Corollary}
\newtheorem{definition}{Definition}
\newtheorem{conjecture}{Conjecture}
\newtheorem{example}{Example}
\newtheorem{claim}{Claim}

\begin{centering}
 
\textup{\LARGE\bf The Yang-Mills flow and the Atiyah-Bott formula on compact K\"ahler manifolds}

\vspace{10 mm}

\textnormal{\large Adam Jacob}

\vspace{.5 in}
\begin{abstract}
{\small

We study the Yang-Mills flow on a holomorphic vector bundle $E$ over a compact K\"ahler manifold $X$. Along a solution of the flow, we show that the curvature endomorphism $i\Lambda F(A_t)$ approaches in $L^2$ an endomorphism with constant eigenvalues given by the slopes of the quotients from the Harder-Narasimhan filtration of $E$. This proves a sharp lower bound for the Hermitian-Yang-Mills functional and thus the Yang-Mills functional,  generalizing to arbitrary dimension a formula of Atiyah and Bott first proven on Riemann surfaces. Furthermore, we show any reflexive extension to all of $X$ of the limiting bundle $E_\infty$  is isomorphic to $Gr^{hns}(E)^{**}$, verifying a conjecture of Bando and Siu.}

\end{abstract}

\vspace{95mm}

\textnormal{ Department of Mathematics,
Harvard University,
Cambridge, MA 02138\\
e-mail: ajacob@math.harvard.edu}\\

\end{centering}

\begin{normalsize}

\newpage
\section{Introduction}
Given a vector bundle $E$ over a compact manifold $X$, the Yang-Mills flow provides a natural approach to constructing Yang-Mills connections on $E$. In addition to their original applications to particle physics, Yang-Mills connections have proven to be useful and important tools in the study of gauge theory and geometry. If $X$ is a smooth $4$-manifold, the moduli space of self-dual Yang-Mills connections reflects deep topological information about $X$ (see $\cite{DonK}$). When $X$ is a compact K\"ahler manifold $X$ of general dimension, if $A$ is a smooth Yang-Mills connection compatible with a given holomorphic structure on $E$, then the curvature endomoprhism $i\Lambda F_A$ will have locally constant eigenvalues determined by the Harder-Narasimhan type of $E$. In fact, such a Yang-Mills connection will decompose $E$ into a direct sum of stable bundles whose slopes corresponds to the slopes of the quotients of the Harder-Narasimhan filtration $\cite{Kob}$. Because of this behavior one would expect existence of Yang-Mills connections to be intimately related to the slope and stability of the original bundle, and this expectation ends up being correct. Throughout this paper we assume $X$ is K\"ahler and $E$ is holomorphic.

If $E$ is indecomposable, a Yang-Mills connection $A$ must be Hermitian-Einstein, meaning that $i\Lambda F_A=\lambda I$ for a topological constant $\lambda$, and it is known such a connection can exists if and only if $E$ is stable in the sense of Mumford-Takemoto. This famous existence theorem for Hermitian-Einstein connections was first proven for curves by Narasimhan and Seshadri $\cite{NS}$, then later reproved by Donaldson in \cite{Don0} following an analytic approach. Donaldson proved the existence of Hermitian-Einstein connections on algebraic surfaces in $\cite{Don1}$, and the result was proven for arbitrary compact K\"ahler manifolds in the celebrated paper of Uhlenbeck and Yau $\cite{UY}$. While Uhlenbeck and Yau proved their theorem using the method of continuity, Donaldson utilized the parabolic approach of the Yang-Mills flow, showing the flow converges to a smooth limit if $E$ is stable. His approach is interesting in that it establishes a relationship between convergence of a parabolic PDE and the algebraic-geometric condition of stability. Utilizing many of the important estimates from $\cite{Don1}$ and $\cite{UY}$, the parabolic approach has been extended to many other cases, most notably to all smooth algebraic varieties by Donaldson $\cite{Don2}$, to Higgs bundles by Simpson $\cite{Simp}$, to reflexive sheaves by Bando-Siu $\cite{BaS}$, as well as an exposition by Siu in $\cite {Siu}$. 

Of course if $E$ is indecomposable and not stable, then the flow can not converge. The main purpose of this paper is to show that nevertheless the limiting properties of the Yang-Mills flow once again reflect many of the geometric properties of $E$, and in many of the same ways as does a Yang-Mills connection. In particular we generalize to arbitrary dimension a theorem of Daskalopoulos and Wentworth from $\cite{DW}$, and we briefly explain their result here.  They show that on a K\"ahler surface $X$, along the Yang-Mills flow the trace of the curvature approaches in $L^p$ an endomorphism with locally constant eigenvalues corresponding to the Harder-Narasimhan type of $E$. Furthermore they prove that away from a bubbling set and along a subsequence, the Yang-Mills flow converges to a limiting Yang-Mills connection on a new bundle $E_\infty$ with a possibly different topology. $E_\infty$ extends over the singular set, and Daskalopoulos and Wentworth prove this extension is isomorphic to the bundle $Gr^{hns}(E)^{**}$, the double dual of the graded quotients of the Harder-Narasimhan-Seshadri filtration. 

This paper builds on previous results of the author \cite{J1, J2} which generalize Daskalopoulos and Wentworth's result to semi-stable bundles over $X$ of arbitrary dimension. Here we extend our work to the general case, in which $E$ is an arbitrary holomorphic vector bundle over $X$. Equip $E$ with a Hermitian metric $H$. Let $Q^i$ be the quotients of the Harder-Narasimhan filtration, and let $\pi^i$ denote the  orthogonal projections onto the subsheaves of this filtration. Define the endomorphism:
\be
\label{endokey}
\Psi_H=\sum_i \mu(Q^i)(\pi^i-\pi^{i-1}).
\ee
This is an endomorphism with locally constant eigenvalues determined by the slopes of the quotients of the Harder-Narsimhan filtration, and because it changes along the Yang-Mills flow, we denote the evolving endomorphism by $\Psi_t$. Although the projections that make up $\Psi_t$ are only smooth where the subsheaves of the Harder-Narasimhan filtration are locally free, we know they are at least in $L^2_1$. We have the following theorem:
\begin{theorem}
\label{firstmaintheorem}
Let $E$ be a holomorphic vector bundle over a compact K\"ahler manifold $X$. Given a fixed metric $H$ and any initial integrable connection $A$ on $E$, let $A_t$ be a smooth solution of the Yang-Mills flow starting with this initial connection. Then for all $\epsilon>0$, there exists a time $t_0$ such that for $t>t_0$, we have
\be
||i\Lambda F_{A_t}-\Psi_{t}||^2_{L^2}<\epsilon.\nonumber
\ee
\end{theorem}
The existence of such a connection for each $\epsilon>0$ is called an $L^2$ approximate Hermitian structure on $E$ (see Definition $\ref{approxHS}$ below). As an immediate consequence we get a sharp lower bound for the Hermitian-Yang-Mills functional $||i\Lambda F(\cdot)||^2_{L^2}$, and since this functional is related to the Yang-Mills functional by a topological constant, we get a sharp lower bound for the Yang-Mills functional as well. In fact we are able to generalize a formula of Atiyah and Bott from $\cite {AB}$. Let $\cal F$ be a slope decreasing filtration of $E$, and let ${\cal Q}^i$ be the quotients of this filtration. Then we define:
\be
\Phi({\cal F})^2=\sum_{i=0}^q\mu({\cal Q}^i)^2rk({\cal Q}^i).\nonumber
\ee
Normalize $\o$ to have volume one, and let $A$ be an integrable connection. We have the following result:
\begin{theorem}
\label{lastcor}
For all holomorphic vector bundles $E$ over $X$ the following formula holds:
\be
\label{AB}
\inf_A||i\Lambda F_A||^2_{L^2}=\sup_{\cal F}\Phi({\cal F})^2.\nonumber
\ee
\end{theorem}
We note that the supremum on the right is attained by the Harder-Narasimhan filtration of $E$. This formula is the higher dimensional generalization of a formula first proven on Riemann surfaces by Atiyah and Bott in $\cite{AB}$. We also direct the reader to the paper of Donaldson $\cite{Don3}$, in which he states the Atiyah-Bott formula and proves a generalization relating the Calabi functional to test configurations.

We now explain our main result, which is an identification of the limit of the Yang-Mills flow. First, given a sequence of connections $A_j$ along the Yang-Mills flow, we define the analytic bubbling set:
\be
Z_{an} =\bigcap_{r>0}\{x\in X\,|\liminf_{j\rightarrow\infty}\,r^{4-2n}\int_{B_r(x)}|F({A_j})|^2\o^n\geq \epsilon\},\nonumber
\ee
for some constant $\e>0$. This set is the same singular set used by Hong and Tian in $\cite{HT}$. Our complete result is as follows:
\begin{theorem}
\label{main theorem}
Let $E$ be a holomorphic vector bundle over a compact K\"ahler manifold $X$. Let $A_t$ be a connection on $E$ evolving along the Yang-Mills flow. Then there exists a subsequence of times $t_j$ such that on $X\backslash Z_{an}$, the sequence $A_{t_j}$ converges (modulo gauge transformations) in $C^\infty$  to a limiting connection $A_\infty$ on a limiting bundle $E_\infty$. $E_\infty$ extends to all of $X$ as a reflexive sheaf $\hat E_\infty$ which is isomorphic to the double dual of the stable quotients of the graded Harder-Narasimhan-Seshadri filtration, denoted $Gr^{hns}(E)^{**}$, of $E$.
\end{theorem}

In $\cite{HT}$, Hong and Tian prove that away from $Z_{an}$, a subsequence along the Yang-Mills flow $A_j$ converges smoothly to a limiting Yang-Mills connection on a limiting bundle $E_\infty$. They also prove that $Z_{an}$ is a holomorphic subvariety of $X$, although we do not utilize this result. By the work of Bando and Siu $\cite{BaS}$, we know $E_\infty$ extends to all of $X$ as a reflexive sheaf $\hat E_\infty$. In this paper we construct an explicit isomorphism between $\hat E_\infty$ and $Gr^{hns}(E)^{**}$, verifying a conjecture of Bando and Siu from $\cite{BaS}$.

Here we remark that the results of this paper are not a full generalization of the work of Daskalopoulos and Wentworth. In $\cite{DW2}$, the authors prove that the bubbling set $Z_{an}$ is in fact equal to the singular set of $Gr^{hns}(E)$, in other words they show the Yang-Mills flow bubbles precisely where the sheaf $Gr^{hns}(E)$ fails to be locally free. Although the results of this paper imply that the singular set of  $Gr^{hns}(E)$ is contained in the analytic singular set $Z_{an}$, the other inclusion does not follow. Sibley and Wentworth are able to accomplish the other inclusion in \cite{SB}, using mostly algebraic methods.

We now briefly describe the proofs of our main results. Our first step is to construct an $L^2$ approximate Hermitian structure on $E$, using a similar method to that of $\cite{J1}$. We begin by defining a new relative functional on the space of Hermitian metrics, denoted $P(H_0,H)$, which is closely related to Donaldson's functional (see $\cite{Don1, Simp, Siu}$). For a fixed metric $H_0$, the $P$-functional is designed so that if $H_t$ is a smooth path of metrics, then the derivative of the $P$-functional along this path is given by:
 \be
 \pl_t\,P(H_0,H_t)=\int_X{\rm Tr}((i\Lambda F-\Psi)H_t^{-1}\pl_t H_t)\o^n.\nonumber
 \ee
We then show that along a solution of the Donaldson heat flow, the following inequality holds:
\be
||i\Lambda F-\Psi||^2_{L^2}\leq- \pl_t\,P(H_0,H_t).\nonumber
\ee
If  $P(H_0,H_t)$ is bounded from below, then $||i\Lambda F-\Psi||^2_{L^2}$ is integrable in time from zero to infinity. This shows $||i\Lambda F-\Psi||^2_{L^2}$ goes to zero along a subsequence, which along with a simple differential inequality proves an $L^2$ approximate Hermitian structure is realized along the Donaldson heat flow. The lower bound the $P$-functional is the difficult step, and is proven in a similar fashion to the lower bound of the Donaldson functional for semi-stable bundles $\cite{J1}$. The key difficulty lies in adapting the blowup procedure from $\cite{J1}$ to regularize the quotients of the Harder-Narasimhan filtration. We show the value of the functional is preserved during this regularization, and take advantage of the fact that on the regularized filtration the $P$-functional decomposes into positive terms plus the sum of the Donaldson functionals on the quotients of the filtration.  We know the Donaldson functional is bounded below on the semi-stable quotients, and thus the $P$-functional is bounded below. 
 
 Once we have established the existence of an $L^2$ approximate Hermitian structure along the Donaldson heat flow, we show such a structure is also realized along the Yang-Mills flow, proving Theorem $\ref{firstmaintheorem}$. Theorem $\ref{lastcor}$ follows. The proof of Theorem $\ref{main theorem}$ requires explicit construction of an isomorphism between $\hat E_\infty$ and $Gr^{hns}(E)^{**}$, which is quite similar to the construction of an isomorphism from $\cite{J2}$ in the case of semi-stable bundles. In our general case, we use Theorem $\ref{firstmaintheorem}$, in combination with a modification of the Chern-Weil formula, to produce the necessary estimates needed for the second fundamental forms associated to the Harder-Narasimhan filtration go to zero in $L^2$. This proves that in the limit we get a holomorphic splitting of $E_\infty$ into a direct sum of semi-stable quotients. Now, utilizing an idea which goes back to Donaldson in $\cite{Don1}$ (and is used by Daskalopoulos and Wentworth in $\cite{DW}$), we can show the holomorphic inclusion maps of the subsheaves from the filtration into $E$ converge to limiting holomorphic maps. Following a stability argument from $\cite{Kob}$ these limiting maps can be shown to be isomorphisms. Fortunately for us much of the hard analysis for this step was carried out by the author in $\cite{J2}$, and we refer the reader to this reference for all relevant details.
 
 The outline of the paper is as follows. In Section 2 we provide preliminary results on holomorphic vector bundles and torsion-free sheaves. We also introduce the Yang-Mills flow, providing important framework for later sections. We introduce the $P$-functional in Section 3, and prove it is bounded from below using the regularization of the Harder-Narasimhan filtration. In Section 4 we prove the existence of an $L^2$ approximate Hermitian structure on $E$, proving Theorem $\ref{firstmaintheorem}$ and Theorem $\ref{lastcor}$. In Section $5$ we construct an isomorphism between $Gr^{hns}(E)^{**}$ and $\hat E_\infty$, proving Theorem $\ref{main theorem}$.

\medskip
\begin{centering}
{\bf Acknowledgements}
\end{centering}
\medskip

First and foremost, the author would like to thank his thesis advisor, D.H. Phong, for all his guidance and support during the process of writing this paper. The author would also like to thank Thomas Nyberg and Tristan Collins for many enlightening discussions. The author thanks Valentino Tosatti for suggesting Theorem $\ref{lastcor}$. Furthermore, the author would like to thank the referee for many helpful comments and suggestions. Finally, the author would like to express his upmost gratitude to Richard Wentworth for encouraging him to work on this problem and providing valuable insight into the structure of the proof. This research was funded in part by the National Science Foundation, Grant No. DMS-07-57372, as well as Grant No. DMS-1204155. The results of this paper are part of the author's Ph.D. thesis at Columbia University.

  \section{Preliminaries}
  \subsection{Vector bundles and natural filtrations}
\label{firstsec}
In this section we introduce our notation and some basic facts about holomorphic vector bundles. Let $X$ be a compact K\"ahler manifold of complex dimension $n$. Locally the K\"ahler form is given by:  
\be
 \o=\frac{i}{2}\,g_{\bar kj}\,dz^j\wedge d\bar z^k,\nonumber
 \ee
where $g_{\bar kj}$ is a Hermitian metric on the holomorphic tangent bundle $T^{1,0}X$. Let $\Lambda$ denote the adjoint of wedging with $\o$. If $\eta$ is a $(1,1)$ form, then in coordinates $\Lambda\eta=-ig^{j\bar k}\eta_{\bar kj}$. The volume form on $X$ is given by $\o^n$, and throughout this paper we normalize $\o$ so that $\int_X\o^n=1$.

Let $E$ be a holomorphic vector bundle over $X$. Given a metric $H$ on $E$, there exists a natural connection called the Chern connection which preserves $H$ and defines the holomorphic structure on $E$. Given a section $\phi$ of $E$, this connection can be written down explicitly in a holomorphic frame:
\be
\nabla_{\bar k}\phi^\al=\pl_{\bar k}\phi^\al\qquad\qquad\nabla_j\phi^\al=\pl_j\phi^\al+H^{\al\bar\b}\pl_j H_{\bar \b\gamma}\phi^\gamma.\nonumber
\ee 
Furthermore, we continue to use $\nabla$ to denote the Chern connection on all associated bundles of $E$, as the bundles we are working with will generally be clear from context. The curvature of $\nabla$ on $E$ is the following endomorphism-valued 2-form:
\be
F:=F_{\bar kj}{}^{\al}{}_{\gamma}\,dz^j\wedge d\bar z^k,\nonumber
\ee
where $F_{\bar kj}{}^{\al}{}_{\gamma}=-\pl_{\bar k}(H^{\al\bar\b}\pl_j H_{\bar \b\gamma})$. Note that $i\Lambda F$ is a Hermitian endomorphism of the bundle $E$, which we denote by $\hat F$ for notational simplicity.  We can now compute the degree of $E$, which we define as follows:
\be
deg(E):=\int_X{\rm Tr}(\hat F)\,\o^n\nonumber.
\ee
We note this definition differs from the usual definition of degree by a factor of $2\pi$, which we omit to make certain formulas simpler later on. Because $X$ is K\"ahler this definition is independent of the choice of metric $H$ on $E$. 
The slope of the vector bundle $E$ is defined to be the following quotient:
\be
\mu(E)=\frac{deg(E)}{rk(E)}.\nonumber
\ee

Given a torsion-free subsheaf ${\cal F}\subset E$, we can view ${\cal F}$ as a holomorphic subbundle off the singular set $Z({\cal F})$ where ${\cal F}$ fails to be locally free. We know from $\cite{Kob}$ that $Z({\cal F})$ is a holomorphic subvariety of $X$ of codimension at least two. Then on $X\backslash Z({\cal F})$ we have a metric on the bundle ${\cal F}$ induced from the metric $H$ on $E$, and the curvature of this metric is at least in $L^1$ $\cite{J1}$. Thus the degree and slope of the subsheaf ${\cal F}$ can be defined in the same way as $E$, by just computing away from the singular set $Z({\cal F})$.

We say $E$ is stable if $\mu({\cal F})<\mu(E)$ for all proper subsheaves ${\cal F}\subset E$ with torsion free quotient. $E$ is semi-stable if $\mu({\cal F})\leq\mu(E)$ for all such ${\cal F}$. Since throughout this paper we have no stability assumptions on $E$, we will need the following proposition, a proof of which can be found in $\cite{Kob}$. 
\begin{proposition}
\label{HNFP}
Any torsion-free sheaf $E$ carries a unique filtration of subsheaves:
\be
\label{HNF}
0=S^0\subset S^1\subset S^2\subset\cdots\subset S^p=E,
\ee
called Harder-Narasimhan filtration of $E$, such that the quotients of this filtration $Q^i=S^i/S^{i-1}$ are torsion-free and semi-stable. The quotients are slope decreasing $\mu(Q^i)>\mu(Q^{i+1})$, and the associated graded object $Gr^{hn}(E):=\bigoplus_iQ^i$ is uniquely determined by the isomorphism class of $E$. 
\end{proposition}
Let $f^i$ denote the holomorphic inclusion of the sheaf $S^i$ into $E$. Also, let $\pi^i$ denote the orthogonal projection of $E$ onto $S^i$ with respect to $H$. We note this projection is defined where $S^i$ is locally free.

We also need an analogous filtration for semi-stable sheaves.  For a torsion-free sheaf ${\cal Q}$ which is semi-stable but not stable, we can always assume there is at least one proper subsheaf $\cal F$ of ${\cal Q}$ such that $\mu({\cal F})=\mu({\cal Q})$. In general there may be many such subsheaves.

\begin{definition}
{\em Given a semi-stable sheaf ${\cal Q}$, a} Seshadri filtration {\rm is a filtration of torsion free subsheaves }
\be
\label{SF1}
0\subset \ti S^0\subset \ti S^1\subset\cdots\subset \ti S^q={\cal Q},
\ee
{\rm such that $\mu(\ti S^i)=\mu({\cal Q})$ for all $i$, and each quotient $\ti Q^i=\ti S^i/\ti S^{i-1}$ is torsion-free and stable. }
\end{definition}
While such a filtration may not be unique, we do have the following proposition, once again from $\cite{Kob}$. 
\begin{proposition}
Given a Seshadri filtration of a torsion free sheaf ${\cal Q}$, the direct sum of the stable quotients, denoted
$Gr^s({\cal Q}):=\bigoplus_{i} \ti Q^i$, is canonical and uniquely determined by the isomorphism class of ${\cal Q}$.
\end{proposition} 

Given our initial holomorphic vector bundle $E$, let $Q^k$ denote the $k$-th quotient of the Harder-Narasimhan filtration. Then $Gr^s(Q^k)$ can be denoted by $\bigoplus_i\ti Q^{i}_k$. Putting these two propositions together, there exists a double filtration of $E$ such that the corresponding graded object:
\be
Gr^{hns}(E):=\bigoplus _k\bigoplus_i \ti Q^{i}_k\nonumber
\ee
is canonical and depends only on the isomorphism class of $E$. We now define the algebraic singular set of $E$ as
\be
Z_{alg}:=\{x\in X\,| Gr^{hns}(E)_x {\rm \,\,is\,\, not\, \,free}\}.\nonumber
\ee
Since the sheaf $Gr^{hns}(E)$ is torsion-free, we know $Z_{alg}$ is of complex codimension at least two. 

Finally, let $r$ be the rank of $E$. We construct an $r$-tuple of real numbers: 
\be
(\mu(Q^1),\cdots,\mu(Q^1),\mu(Q^2),\cdots,\mu(Q^2),\cdots,\mu(Q^p),\cdots,\mu(Q^p)),\nonumber
\ee
where the multiplicity of each number $\mu(Q^i)$ is given by $rk(Q^i)$. We call this $r$-tuple the Harder-Narasimhan type of $E$. Now, recall from $\eqref{endokey}$ the endomorphism $\Psi_H$, whose eigenvalues are defined to be the Harder-Narasimhan type of $E$. We note the dependence on the metric $H$ comes from metric dependence on the orthogonal projections $\pi^i:E\longrightarrow S^i$.  
\begin{definition}
\label{approxHS}
We say $E$ carries an $L^p$ approximate Hermitian structure if for all $\epsilon>0$, there exists a metric $H$ on $E$ such that:
\be
||\hat F-\Psi_H||_{L^p}<\epsilon.\nonumber
\ee
\end{definition}

\subsection{Decomposition onto subsheaves}
\label{bundledecomp}
In this subsection we address how the curvature behaves on subsheaves of $E$. Let $S\subset E$ be a proper, torsion-free subsheaf, which we include in the following short exact sequence:
\be
\label{sequence}
0\longrightarrow S\xrightarrow{{\phantom {X}}{f}{\phantom {X}}} E\xrightarrow{{\phantom {X}}p{\phantom {X}}} Q\longrightarrow 0,
\ee
where we assume that the quotient sheaf $Q$ is torsion free. Define the singular set of $Q$ by $Z(Q):=\{x\in X\,|\, Q_x$ is not free$\}$. Then on $X\backslash Z(Q)$, we can view $\eqref{sequence}$ as a short exact sequence of holomorphic vector bundles.  Here, a smooth metric $H$ on $E$ induces a metric ${J}$ on $S$ and a metric $K$ on $Q$. For sections $\psi,\phi$ of $S$, we define the metric $J$ as follows:
\be
\langle\phi,\psi\rangle_J=\langle f(\phi),f(\psi)\rangle_H.\nonumber
\ee
In order to define the smooth metric $K$ on $Q$, we note that the choice of $H$ on $E$ defines a splitting of $\eqref{sequence}$:
\be
\label{splitting}
0\longleftarrow{S}\xleftarrow{{\phantom {X}}{\pi}{\phantom {X}}} {E}\xleftarrow{{\phantom {X}}{p^\dagger}{\phantom {X}}} {Q}\longleftarrow0.
\ee
For sections $v,w$ of $Q$, we define the metric $K$ as follows:
\be
\langle v,w\rangle_K=\langle p^{\dagger}(v),p^{\dagger}(w)\rangle_H.\nonumber
\ee
\begin{definition}
{\rm On $X\backslash Z(Q)$ both $S$ and $Q$ are holomorphic vector bundles. We define an} induced metric {\rm on either $S$ or $Q$ to be one constructed as above.} 
\end{definition}
We emphasize that induced metrics are not defined on all of $X$, and they may may degenerate or blow up as we approach the singular set, causing curvature terms to blow up.

Once we have sequence $\eqref{splitting}$, the second fundamental form $\gamma\in\Gamma(X,\Lambda^{0,1}\otimes Hom(Q,S))$ is given by:
\be
\gamma:= \pi\circ\bar\pl\circ p^\dagger. \nonumber
\ee 
Although the second fundamental form is not defined on all of $X$,  in $\cite{J1}$ it is shown that $\gamma$ is at least in $L^2(X)$. We now derive a formula for the second fundamental form in terms of $\pi$. As $\eqref{splitting}$ shows, $\pi$ is  the orthogonal projection from $E$ onto $S$. First, note that in fact $\bar\pl\circ p^\dagger$ already lies in $S$, since for any $q\in\Gamma(X\backslash Z(Q),Q)$, $p$ is holomorphic and $p\circ p^\dagger=I$, thus $p\,(\bar\pl\circ p^\dagger(q))=0$. Now, because $p^\dagger\circ p=I-\pi$, we have $\gamma\circ p=\bar\pl (p^\dagger)\circ p=\bar\pl(p^\dagger\circ p)=\bar\pl(I-\pi)=-\bar\pl\pi$. Thus $||\gamma||^2_{L^2}=||\bar\pl\pi||^2_{L^2}$, and $\pi\in L^2_1$. Conversely, as proven by Uhlenbeck and Yau in $\cite{UY}$, any weakly holomorphic $L^2_1(X)$ projection defines a coherent subsheaf of $E$ (see Popovici \cite{Po} for a simplified proof of this result). Thus later on in the paper we will go back and forth between working with a subsheaf $S$ and the $L^2_1(X)$ projection $\pi$ that defines the subsheaf.

We now turn to the decomposition of connections and curvature onto subbundles and quotient bundles, which is described in detail in $\cite{GH}$. Because of their prominence throughout the paper, we review some of these decomposition formulas here. We continue to work on $X\backslash Z(Q)$. Let $\nabla^S$ and $\nabla^Q$ be the Chern connections on $S$ and $Q$ with respect to the metrics $J$ and $K$. In a local coordinate patch, any section $\Phi$ of $E$ decomposes  onto the bundles $S$ and $Q$, denoted $\Phi=\phi+q$. We now have the following decomposition of $\nabla$:
\be
\label{connectiondecomp}
\nabla (\Phi)=\left( 
\begin{array}{cc}
\nabla^S & \gamma\\
-\gamma^\dagger & \nabla^Q \end{array} 
\right)
\left(
\begin{array}{c}
\phi \\
q\end{array} 
\right).
\ee
Now, denote the curvature of the induced metric $J$ by $F^S$ and the curvature of the induced metric $K$ by $F^Q$. The full curvature tensor $F$ decomposes as follows:
\be
\label{decompcurv}
F(\Phi)=\left( 
\begin{array}{cc}
F^S-\gamma\wedge\gamma^\dagger & \nabla\gamma\\
-(\nabla\gamma)^\dagger & F^Q-\gamma^\dagger\wedge\gamma \end{array} 
\right)
\left(
\begin{array}{c}
\phi \\
q\end{array} 
\right).
\ee
\subsection{The Yang-Mills flow}

In this section we describe our approach to the Yang-Mills flow. We follow the viewpoint taken by Donaldson in $\cite{Don1}$, which relates the flow of a metric in a fixed holomorphic structure to the evolution of an integrable unitary connection. This relationship is clearly explained in $\cite{Don1}$, and we direct the reader there for details. Here we simply present the setup and include the important facts needed for the arguments to follow. 

Fix a metric $H_0$ on $E$. Let $d_A$ be a unitary connection with local connection matrix $A$. Since $X$ is complex, $d_A$ will decompose into $(1,0)$ and $(0,1)$ parts, which we denote by $\pl_A=\pl+A'$ and $\bar\pl_A=\bar\pl+A''$. We say $d_A$ is integrable if $\bar\pl_A^2=0$ (thus $d_A$ defines a holomorphic structure by the Newlander-Nirenberg integrability theorem), and we denote the space of integrable unitary connections by ${\cal A}^{1,1}$. The curvature 2-form of such a connection only has a $(1,1)$ component, and is denoted by $F_A$. The Yang-Mills functional $YM:{\cal A}^{1,1}\longrightarrow \R$ can now be expressed:
\be
YM(A):=||F_A||^2_{L^2}.\nonumber
\ee
Now, on a general compact manifold, the Yang-Mills flow is the gradient flow of this functional, given by:
\be
\dot A=-d_A^*\,F_A.\nonumber
\ee
However, because we are on a K\"ahler manifold, Bianchi's second identity ($d_AF_A=0$) and the K\"ahler identities allow us to express the Yang-Mills flow in a simpler form:
\be
\label{YMF2}
\dot A=i\bar\pl_A\Lambda F_A-i\pl_A\Lambda F_A.
\ee
From this formulation one can check that the Yang-Mills flow stays inside ${\cal A}^{1,1}$ if we start with an integrable connection. 

Donaldson defines a flow of metrics with respect to a fixed holomorphic structure $\bar\pl_{A_0}$ in order to further study the Yang-Mills flow. Given $H_0$, any other metric $H$ defines a positive definite Hermitian endomorphism $h$ by the formula $h=H_0^{-1} H$. The Donaldson heat flow is a flow of endomorphisms $h=h(t)$ given by:
\be
\label{DHF}
h^{-1}\dot h=-(\hat F-\mu(E) I),
\ee
where $F$ is the curvature of the Chern connection of the metric $H(t)=H_0 h(t)$ and the holomorphic structure $\bar\pl_{A_0}$. Here the Einstein constant is simply given by $\mu(E)$ since we normalized $X$ to have volume one. Setting the initial condition $h(0)=I$, a unique smooth solution of the flow exists for all $t\in[0,\infty)$, and on any stable bundle this solution will converge to a smooth Hermitian-Einstein metric $\cite{Don1, Don2, Simp, Siu}$.

In our case the bundle $E$ is not stable, so we do not expect the flow to converge. However, we can use a solution to \eqref{DHF} to construct a solution to the Yang-Mills flow. Let $d_{A_0}$ be an initial connection in ${\cal A}^{1,1}$. We consider the flow of holomorphic structures
$\bar\pl_t=\bar\pl+A_t'',\nonumber$
where $A_t''$ is defined by the action of $w=h^{1/2}$ on $A_0''$. Explicitly, this action is given by:
\be
\label{action}
A_t''=wA_0''w^{-1}-\bar\pl ww^{-1}.
\ee
Define a flow of connections $d_{A_t}$ by imposing the unitary condition with respect to $H_0$ on the above flow of holomorphic structures. This flow is gauge equivalent to a solution of the Yang-Mills flow. Conversely, any path in ${\cal A}^{1,1}$ along the Yang-Mills flow defines an orbit of the complexified gauge group, which gives rise to a solution of the Donaldson heat flow (see $\cite{Don1}$ for details). As a result we can go back and forth between the two flows. Note that given this setup, the curvature  $F$ along the Donaldson heat flow is related to the curvature $F_A$ along the Yang-Mills flow by the following relation:
\be
\label{GRL}
F_A=w \,F\,w^{-1}.
\ee

We conclude this section by stating the convergence result of Hong and Tian from $\cite{HT}$. Consider a sequence of connections $A_j$ evolving along the Yang-Mills flow. Then, on $X\backslash Z_{an}$, along a subsequence the connections $A_j$ converge in $C^\infty$, modulo unitary gauge transformations, to a Yang-Mills connection $A_\infty$. Thus, always working on $X\backslash Z_{an}$,  we have a sequence of holomorphic structures $(E,\bar\pl_j)$ which converge in $C^\infty$ to a holomorphic structure $(E,\bar\pl_\infty)$. By the work of Bando and Siu, the bundle $(E,\bar\pl_\infty)$ extends to all of $X$ as a reflexive sheaf $\hat E_\infty$. Once again the main goal of this paper is to identify $\hat E_\infty$ with $Gr^{hns}(E)^{**}$, proving this limit is canonical and independent of subsequence.

\section{The $P$-functional}
We now begin the proof of Theorem $\ref{firstmaintheorem}$, starting with the construction of an $L^2$ approximate Hermitian structure on $E$. First we introduce the $P$-functional and describe some basic properties. 

Fix an initial metric $H_0$ on $E$. Then for any other metric $H$ we can define the endomorphism $h=H^{-1}_0H$. Consider any path $h_t$, $t\in[0,1]$, of positive definite Hermitian endomorphisms such that $h_0=I$ and $h_1=h$. The $P$-functional is defined by:
\be
P(H_0,H)=\int_0^1\int_X{\rm Tr}((\hat F_t-\Psi_t)h_t^{-1}\dot h_t)\,\o^n\,dt,\nonumber
\ee
where $F_t$ is the curvature of the metric $H_t=H_0h_t$. The above integral is well defined, for even though the projections $\pi^i$ that make up $\Psi_t$ are only defined on $X\backslash Z_{alg}$, we know that they are at least in $L^2_1$ $\cite{J1}$. Recall from Section \ref{firstsec} that $f_i$  denotes the holomorphic inclusion of $S^i$ into $E$. Although we generally view the projection $\pi^i$ as an endomorphism of $E$, its image is isomorphic to $S^i$ via the holomorphic inclusion $f_i$, and on occasion we implicitly make use of this fact. We now check the $P$-functional is independent of path.

\begin{proposition}
\label{path}
The $P$-functional is path independent for any pair of metrics $H_0, H$ on $E$.
\end{proposition}
\begin{proof}
We note that the first term
\be
\int_0^1\int_X{\rm Tr}(\hat F_th_t^{-1}\dot h_t)\,\o^n\,dt,\nonumber
\ee
appears in the Donaldson functional and is shown to be path independent in Chapter 1, Section 5, of $\cite{Siu}$. Therefore we turn our attention to the second term:
\be
\int_0^1\int_X{\rm Tr}(\Psi_th_t^{-1}\dot h_t)\,\o^n\,dt=\sum_i\mu(Q^i)\int_0^1\int_X{\rm Tr}((\pi^i_t-\pi^{i-1}_t)h_t^{-1}\dot h_t)\,\o^n\,dt.\nonumber
\ee
Note that ${\rm Tr}(\pi^i_th_t^{-1}\dot h_t)={\rm Tr}(\pi^i_th_t^{-1}\dot h_tf^i\pi^i_t)={\rm Tr}(\pi^i_th_t^{-1}\dot h_tf^i)$, where $\pi^i_th_t^{-1}\dot h_tf^i$ is now an endomorphism of the bundle $S^i$. We need the following lemma.
\begin{lemma}
\label{h^{-1}h}
Dropping the subscript $t$ for simplicity, we have:
\be
\pi^ih^{-1}\dot hf^i=(h^i){}^{-1}\dot h^i,\nonumber
\ee
where $J_0^i$, $J^i$ are the induced metrics on the subbundle $S^i$ defined by $H_0$ and $H$, and $h^i$ is the endomorphism of $S^i$ defined by $h^i=(J_0^i){}^{-1}J^i$.
\end{lemma}
\begin{proof}
First we note that $h^{-1}\dot h$ can be defined using the derivative of the metric $H$:
\be
\pl_t \langle\cdot,\cdot,\rangle_H= \pl_t\langle h(\cdot),\cdot,\rangle_{H_0}= \langle\dot h(\cdot),\cdot,\rangle_{H_0}=\langle h^{-1}\dot h(\cdot),\cdot,\rangle_H.\nonumber
\ee
Thus for any two sections $\psi,\phi$ of $S^i$, we define $(h^i){}^{-1}\dot h^i$ by:
\be
\pl_t \langle\psi,\phi\rangle_{J^i}=\langle (h^i){}^{-1}\dot h^i\psi,\phi\rangle_{J^i}.\nonumber
\ee
However by definition of the induced metric we have
\be
\pl_t \langle\psi,\phi\rangle_{J^i}=\pl_t \langle f^i\psi,f^i\phi\rangle_{H}=\langle h^{-1}\dot hf^i\psi,f^i\phi\rangle_{H}=\langle \pi^ih^{-1}\dot hf^i\psi,\phi\rangle_{J^i},\nonumber
\ee
concluding the lemma.
\end{proof}
Of course the lemma is only true where $S^i$ is locally free, thus we restrict ourself to $X\backslash Z_{alg}$. On this set we have ${\rm Tr}(\pi^i_th_t^{-1}\dot h_t)={\rm Tr}((h^i_t){}^{-1}\dot h^i_t)=\pl_t$log det$(h^i_t),$ so
\bea
\int_0^1\int_X{\rm Tr}(\pi^i_th_t^{-1}\dot h_t)\,\o^n\,dt&=&\int_0^1\int_{X\backslash Z_{alg}}{\rm Tr}(\pi^i_th_t^{-1}\dot h_t)\,\o^n\,dt\nonumber\\
&=&\int_0^1\pl_t\int_{X\backslash Z_{alg}}{\rm log\,det}(h^i_t)\,\o^n\,dt\nonumber\\
&=&\int_{X\backslash Z_{alg}}{\rm log\,det}(h^i_1)\,\o^n.\nonumber
\eea
Thus the integral is path independent.
\end{proof}

The goal of the next few subsections is to prove the following theorem:
\begin{theorem}
\label{lowerb}
For a fixed reference metric $H_0$, the functional $P(H_0,H)$ is bounded below for all other Hermitian metrics $H$. 
\end{theorem}
This theorem is major step in the proof of Theorem $\ref{firstmaintheorem}$. As a first step towards its proof we must regularize the Harder-Narasimhan filtration.

\subsection{Regularization of the Harder-Narasimhan filtration}
\label{RHNF}
In $\cite{J1}$, the author employs a procedure to regularize a torsion free subsheaf of $E$. In this section we describe that result, and explain how it can be easily expanded to regularize any filtration of subsheaves of $E$. We note that the following procedure is consistent with a viewpoint found in Uhlenbeck and Yau $\cite{UY}$. In their paper they view a torsion free sheaf locally as a rational map from $X$ to the Grassmanian $Gr(s,r)$. By Hironaka's Theorem we know this map can be regularized after a finite number of blowups. We follow our procedure because it lets us keep track of how that map changes in local coordinates at each step, which is important in the analysis that follows.

First we recall the result from $\cite{J1}$. Consider the short exact sequence of sheaves $\eqref{sequence}$. Here $E$ is locally free and $Q$ torsion free. Suppose $S$ has rank $s$, $E$ has rank $r$, and $Q$ has rank $q$. After choosing coordinates, off $Z(Q)$ we view $f$ as an $r\times s$ matrix of holomorphic functions with full rank. These matrices transform by given transition functions on the coordinate overlaps. As one approaches $Z$ the rank of $f$ may drop, and it is exactly this behavior that needs to regularized.

Let $Z_k$ be the subset of $Z(Q)$ where $rk(f)\leq k$. For the smallest $k$ such that $Z_k$ is nonempty, at a point we can choose coordinates so that $f$ can be expressed as 
\be
f=\left(
\begin{array}{cc}I_k&0\\0&g
 \end{array} 
\right),\nonumber
\ee
where $g$ vanishes identically on $Z_k$. Blow up along $Z_k$ by the map $\pi:\ti X\longrightarrow X$. On a given coordinate patch of $\ti X$ let $w$ define the exceptional divisor. Then the pullback of $f$ can be decomposed as follows:
\be
\label{decomf}
\pi^*f=\left(
\begin{array}{cc}I_k&0\\0&\ti g
 \end{array} 
\right)\left(
\begin{array}{cc}I_k&0\\0&w^aI_{s-k}
 \end{array} 
\right),
\ee
where $a$ is the largest power of $w$ we can pull out of $\pi^*g$. Denote the matrix on the left of $\eqref{decomf}$ as $\ti f$ and the matrix on right as $t$. In $\cite{J1}$ it is shown that the map $\ti f$ defines a new torsion free subsheaf $\ti S$ of $\pi^*E$ by explicitly writing down transition functions. Furthermore it is shown that this procedure (applied to each $Z_k$) stops after a finite number of blowups to produce a map where the rank does not drop anywhere on $X$, and thus defines a holomorphic subbundle of $\pi^*E$.

We now turn our attention to the Harder-Narasimhan filtration of $E$ $\eqref{HNF}$. Recall that $f^i:S^i\longrightarrow E$ denotes the holomorphic inclusion of $S^i$ into $E$, and let $l^i:S^{i}\longrightarrow S^{i+1}$ be the holomorphic inclusion of each subsheaf $S^i$ into the corresponding sheaf of next lowest rank $S^{i+1}$. Then we have that $f^{p-1}=l^{p-1}$, $f^{p-2}=l^{p-1}\circ l^{p-2}$, and in general $f^i=l^{p-1}\circ\cdots\circ l^i$. To regularize this filtration, we begin by regularizing each subsheaf, starting with $S^1$ and then working with subsheaves of successively higher rank. We describe the process as follows.

Given $S^i$ from the filtration, for each $i$ we have a sequence of blowups $\pi_i:\ti X^i\longrightarrow \ti X^{i-1}$ and a corresponding holomorphic inclusion map $\ti f^i:\ti S^i\longrightarrow {\pi_i}^*E$ such that the rank of $\ti f^i$ does not drop. From $\eqref{decomf}$ we know that locally $\ti f^i$ is defined by ${\pi_i}^*f^i=\ti f^i\circ t$, where $t$ is some diagonal matrix of monomials of sections defining the exceptional divisor. Since $f^i=l^{p-1}\circ\cdots\circ l^i$, and $\ti f^i={\pi_i}^*f^i\circ t^{-1}$, we can define $\ti l^i:={\pi_i}^*l^i\circ t^{-1}$. Because after a finite number of steps the rank of $\ti f^i$ does not drop, we have now that the rank of  $\ti l^i$ does not drop, thus
\be
0\longrightarrow (\pi_{i+1})^*\ti S^{i}\xrightarrow{{\phantom {X}}{\ti l^i}{\phantom {X}}} {\ti S^{i+1}}\nonumber
\ee
defines a holomorphic inclusion of vector bundles, and the regularized quotient $\ti Q^{i+1}$ is a holomorphic vector bundle. Following this construction for all $i$ we have a finite sequence of blowups that regularizes each sheaf in the Harder-Narasimhan filtration of $E$, such that the quotients $\ti Q^i$ are all locally free. Of course, we never used semi-stability of the quotients, so our procedure applies to any filtration of sheaves on $X$. Summing up we have proved the following proposition:
\begin{proposition}
\label{regularization}
Given a holomorphic vector bundle $E$ over a compact, complex manifold $X$, let 
\be
0=S^0\subset S^1\subset S^2\subset\cdots\subset S^p=E\nonumber
\ee
be a filtration of $E$ by subsheaves. Away from the singular sets the inclusion maps $l^i_0:S^i\longrightarrow S^{i+1}$ can be defined locally by matrices of holomorphic functions with transition functions on the overlaps. There exists a finite number of blowups
\be
\ti X_N\xrightarrow{{\phantom {X}}{\pi_N}{\phantom {X}}}\ti X_{N-1}\xrightarrow{{\phantom {X}}{\pi_{N-1}}{\phantom {X}}}\cdots \xrightarrow{{\phantom {X}}\pi_2{\phantom {X}}}\ti X_1\xrightarrow{{\phantom {X}}\pi_1{\phantom {X}}}X,\nonumber
\ee
and matrices of holomorphic functions $l^i_k$ over $\ti X_k$ with the following properties:

\medskip
i) On each $\ti X_k$ around a given point there exists coordinates so that if $w$ defines the exceptional divisor, there exists a diagonal matrix of monomials in $w$ (denoted $t$) so that
\be
\pi^*_{k-1} l^i_{k-1}= l^i_k\circ t.\nonumber
\ee

\medskip
ii) The rank of $l^i_N$ is constant for each $i$, thus it defines a holomorphic subbundle of $\ti S^{i+1}$ with a holomorphic quotient bundle. 
\end{proposition}

\subsection{Transformation of key terms}

We now turn our attention back to Theorem $\ref{lowerb}$. We prove this theorem by changing the form of the $P$-functional and writing it as a sum of objects which we know are bounded from below. First we recall the definition of the Donaldson functional on a vector bundle $E$:
\be
M(H_0,H,\o)=\int_0^1\int_X{\rm Tr}(F_th^{-1}_t\pl_t h_t)\wedge\o^{n-1}\,dt-{\mu(E)}\int_X{\rm log\,det}(h_1)\,\o^n,\nonumber
\ee
where once again $h_t$ is any path of positive definite Hermitian matrices with $h_0=I$ and $h_1=H_0^{-1}H$, and $F_t$ is the curvature of the metric $H_t:=H_0\,h_t$ along the path. Here we introduced $\o$ as input into the functional to show its dependence on a volume form. Now, in the analysis to follow it will be show that the Donaldson functional is well defined on holomorphic subsheaves and holomorphic quotient sheaves of $E$. As a result let $M_i(H_0,H,\o)$ denote the Donaldson functional on the quotient sheaf $Q^i$ with induced metrics from $E$ (this quantity is defined explicitly in Definition \ref{Donfunc}). We will prove that:
\be
\label{decomp}
P(H_0,H)=\sum_i\,M_i(H_0,H,\o)+||\gamma^i||_{L^2}^2-||\gamma^i_0||_{L^2}^2.
\ee
Here $\gamma^i$ is the second fundamental form of the short exact sequence:
\be
0\longrightarrow S^{i-1}\longrightarrow S^i\longrightarrow Q^i\longrightarrow 0,\nonumber
\ee
associated to the metric $H$, and $\gamma^i_0$ is the second fundamental form associated to $H_0$. Thus to prove Theorem $\ref{lowerb}$ we have to complete two steps. First we show that all the terms in $\eqref{decomp}$ are well defined for induced metrics on the sheaves $Q^i$, and second we need to show that the functional does indeed satisfy the decomposition $\eqref{decomp}$. In this subsection we will focus on showing all the terms are well defined.

As in the previous subsection much of the analysis we need has been carried out in $\cite{J1}$. We refer the reader to that reference for the details of the proofs, and here present the results, modified to our specific case. From Section $\ref{RHNF}$ we recall that there exists a regularized Harder-Narasimhan filtration
\be
0=\ti S^0\subset\ti S^1\subset\cdots\subset \ti S^{p-1}\subset \ti S^p=\pi^*E,\nonumber
\ee
such that the rank of the holomorphic inclusion maps $\ti f^i$ does not drop on $\ti X$ (here $\pi:\ti X\longrightarrow X$ is the sequence of blowups needed to construct the regularization). So given $\pi^*H$ on $\pi^*E$, the smooth induced metric on $\ti S^i$ is defined by:
\be
\ti J^i_{\bar \b\al}:= (\ti f^i)^\rho{}_\al\overline{(\ti f^i)^\gamma{}_\b}\,\pi^*H_{\bar\gamma\rho}.\nonumber
\ee
Also, because the rank of $\ti l^i:\ti S^i\longrightarrow\ti S^{i+1}$ does not drop, we have an exact sequence of holomorphic vector bundles:

\be
\label{seq1}
0\longrightarrow \ti S^i\xrightarrow{{\phantom {X}}{\ti l^i}{\phantom {X}}} \ti S^{i+1}\xrightarrow{{\phantom {X}}p^i{\phantom {X}}} \ti Q^{i+1}\longrightarrow 0.
\ee
The metric $\ti J^{i+1}$ gives a splitting of the short exact sequence:
\be
\label{seq2}
0\longleftarrow{\ti S^i}\xleftarrow{{\phantom {X}}{\lambda^{i}}{\phantom {X}}} {\ti S^{i+1}}\xleftarrow{{\phantom {X}}{{p^i}^\dagger}{\phantom {X}}} {\ti Q^{i+1}}\longleftarrow0,
\ee
and it follows that the metric $\ti K_{\bar \b\al}$ on $\ti Q^{i+1}$ defined by:
\be
\ti K_{\bar\b\al}:=({p^i}^\dagger)^\rho{}_\al\overline{({p^i}^\dagger)^\gamma{}_\b}\ti J^i_{\bar\gamma\rho}\nonumber
\ee
is smooth. 

Our main concern is that the integrals that make up each term in $\eqref{decomp}$ might not be finite, which is a reasonable concern because along $Z_{alg}$, curvature terms will blow up. We show these terms are controlled by using formulas describing the change during each step in the regularization procedure, and prove that in fact the desired terms do not change during regularization. Once we are working with the regularized filtration the induced metrics are smooth, and since the manifold is compact each term will be finite.

We recall the following proposition from $\cite{J1}$:
\begin{proposition}
Consider a single blowup from the regularization procedure $\pi:\ti X\longrightarrow X$. Let $J$ and $K$ be induced metrics on $S^i$ and $Q^i$, respectively. Then if $w$ locally defines the exceptional divisor $D$, there exist non-negative integers $a_\al$ so that:
\be
\pi^*{J}_{\bar\b\al}=w^{a_\al}\overline{w^{a_\b}}\ti{J}_{\bar\b\al}\qquad \qquad
\pi^*K_{\bar\b\al}=\frac{1}{w^{a_\al}\overline{w^{a_\b}}}\ti K_{\bar\b\al}.\nonumber
\ee 
\end{proposition}
Using this proposition one can compute how the induced curvature changes during each blowup, and we include the computation for the reader's convenience. Let $F$ be the curvature of the quotient sheaf $Q^i$. We work in a local trivialization and apply the previous proposition:
\bea
\pi^*F_{\bar kj}{}^\al{}_\b&=&-\pl_{\bar k}(\pi^*K^{\al\bar\gamma}\pl_j\pi^*K_{\bar \gamma\b})\nonumber\\
&=&-\pl_{\bar k}(\ti K^{\al\bar\gamma}w^{a_\al}\overline{ w^{a_\gamma}}\pl_j(\frac{1}{w^{a_\b}\overline{w^{a_\gamma}}}\ti K_{\bar\gamma\b})).\nonumber
\eea
Now since $\overline w^{a_\gamma}$ is anti-holomorphic, it follows that
\bea
\pi^*F_{\bar kj}{}^\al{}_\b&=&-\pl_{\bar k}(\ti K^{\al\bar\gamma}w^{a_\al}\pl_j(\frac{1}{w^{a_\b}}\ti K_{\bar\gamma\b}))\nonumber\\
&=&-\pl_{\bar k}(w^{a_\al}\pl_j(\frac{1}{w^{a_\b}})\ti K^{\al\bar\gamma}\ti K_{\bar \gamma\b}+\frac{w^{a_\al}}{w^{a_\b}}\ti K^{\al\bar\gamma}\pl_j \ti K_{\bar\gamma\b})\nonumber\\
&=&a_\al\pl_j\pl_{\bar k}{\rm log}|w|^2\delta^{\al}{}_\b-\pl_{\bar k}(\frac{w^{a_\al}}{w^{a_\b}}\ti K^{\al\bar\gamma}\pl_j \ti K_{\bar\gamma\b}).\nonumber
\eea
Note that the first term on the left vanishes away from $D$. This computation has two important corollaries, which we now state. For simplicity we restrict ourselves to working with $K$ on the quotient $Q^i$, noting that a similar formula holds for induced metrics on $S^i$. The first corollary follows from the Poincar\'e-Lelong formula.

\begin{corollary}
\label{cordegree}
Consider a single blowup from the regularization procedure $\pi:\ti X\longrightarrow X$, and let $D$ be the exceptional divisor. Then the following decomposition holds in the sense of currents
\be
\pi^*{\rm Tr} (F)=\frac{2\pi}i\left(\sum_\al a_\al\right) [D]+{\rm Tr(\ti F)}.\nonumber
\ee
\end{corollary}
The following corollary is also proven in detail in $\cite{J1}$.
\begin{corollary}
The $L^2$ norm of the second fundamental form is well defined for any torsion free subsheaf $S^i$ of $S^{i+1}$ with torsion free quotient $Q^i$. Furthermore, for a given blowup from the regularization procedure $\pi:\ti X\longrightarrow X$, we have the following equality
\be
||\gamma^i|| ^2_{L^2(\o)}=||\ti\gamma^i||^2_{L^2(\pi^*\o)},\nonumber
\ee
where $\gamma^i$ is the second fundamental form associated to $S^i$, and $\ti\gamma^i$ to $\ti S^i$.
\end{corollary}
The proof of the preceding corollary uses Corollary \ref{cordegree} to show that the $L^2$ norm of $\gamma$ does not change during each step of the regularization. Because at the final step all the induced metrics are smooth, the integral is of a smooth function over a compact manifold, and thus is well defined. 

Next we show that the Donaldson functional $M_i(H_0,H,\o)$ is well defined on any quotient sheaf $Q^i$ arising from the filtration.  Given a blowup map $\pi:\ti X\longrightarrow X$, one can also define the Donaldson functional on a vector bundle over $\ti X$ by integrating with respect to the degenerate metric $\pi^*\o$. Since $\pi^*\o$ is closed the functional will still be independent of path. We define the Donaldson functional on the sheaves $Q^i$ as follows:

\begin{definition}
\label{Donfunc}
{\rm For any quotient sheaf $Q^i$ arising from the Harder-Narasimhan filtration of $E$, we define the} Donaldson functional {\rm on $Q^i$ to be:}
\be
M_i(H_0,H,\o):=M_{\ti Q}(\ti K_0,\ti K,\pi^*\o),\nonumber
\ee
{\rm for any regularization $\ti Q^i$.}
\end{definition}
Here $M_{\ti Q}(\ti K_0,\ti K,\pi^*\o)$ is the Donaldson functional for the vector bundles $\ti Q$ defined using the degenerate metric $\pi^*\o$. We note that the domains of the functionals $M_i$ are metrics on the vector bundle $E$, thus this definition only applies to induced metrics and does not extend to arbitrary metrics on $Q^i$. The following proposition proves that this definition is well defined.
\begin{proposition}
\label{donpres}
For each $i$ the functional $M_i$ is well defined for any pair of metrics on $E$, and is independent of the choice of regularization.
\end{proposition}
The proof of Proposition \ref{donpres} again rests on our computation of $\pi^*F$. One can check that after each blow-up in the regularization procedure the value of the Donaldson functional remains the same. We direct the reader to $\cite{J1}$ for details. Immediately we see the $P$-functional is well defined on the subsheaves $S^i$ as well, and that its value is independent of regularization. Now all three terms on the right hand side of $\eqref{decomp}$ are well defined for induced metrics on the quotient sheaves $Q^i$. The next step is to show that the decomposition formula does indeed hold.

\subsection{Decomposition of the P-functional}

In this section we prove decomposition formula $\eqref{decomp}$, using an argument similar to Donaldson $\cite{Don1}$. We begin by considering the proper subsheaf of highest rank in the filtration, $S^{p-1}$. In the proof of Proposition $\ref{path}$, we found the following formula for $P$:
\be
P(H_0,H)=\int_0^1\int_X{\rm Tr}(\hat F_th_t^{-1}\dot h_t)\,\o^n\,dt-\sum_i\mu(Q^i)\int_X({\rm log\,det}(h^i_1)-{\rm log\,det}(h^{i-1}_1))\,\o^n,\nonumber
\ee
where $h^i$ is the endomorphism defined by induced metrics $J^i$ and $J^i_0$ on $S^i$. Dropping the subscript $t$, we note that by Proposition $\ref{donpres}$ we have:
\be
\int_0^1\int_X{\rm Tr}(\hat F h^{-1}\dot h)\o^n=\int_0^1\int_{\ti X}{\rm Tr}(\pi^*(\hat F h^{-1}\dot h))\pi^*\o^n,\nonumber
\ee
where $\pi:\ti X\longrightarrow X$, is a sequence of blowups which regularizes the Harder-Narasimhan filtration. Now, the regularized $\ti S^{p-1}$ and $\ti Q^{p}$ are holomorphic subbundles and quotient bundles of $\pi^*E$, and with the metric $\pi^*H$ we can identify the following splitting:
\be
\label{newsplitting}
0\longleftarrow{\ti S^{p-1}}\xleftarrow{{\phantom {X}}{}{\phantom {X}}} {\pi^*E}\xleftarrow{{\phantom {X}}{p^\dagger}{\phantom {X}}} {\ti Q^p}\longleftarrow0.\nonumber
\ee
Following Section $\ref{bundledecomp}$ we have the decomposition of curvature:
\be
\pi^*\hat F=
\left( 
\begin{array}{cc}
\hat F^{\ti S^{p-1}}+ \pi^*g^{j\bar k}\gamma_{\bar k}\gamma^\dagger_j & \pi^*g^{j\bar k}\nabla_j\gamma_{\bar k}\\
\pi^*g^{j\bar k}\nabla_{\bar k}\gamma^\dagger_j & \hat F^{\ti Q^p}-\pi^*g^{j\bar k}\gamma^\dagger_j\gamma_{\bar k} \end{array} 
\right)
.\nonumber
\ee
Define $V:=p^\dagger-p_0^\dagger$. Using the description of $h^{-1}\dot h$ from the proof of Lemma $\ref{h^{-1}h}$ we can see how $h^{-1}\dot h$ decomposes:
\be
 h^{-1}\dot h=
\left( 
\begin{array}{cc}
(h^{-1}\dot h)^{p-1}& -\dot V\\
-\dot V^\dagger & (h^{-1}\dot h)^p \end{array} 
\right).\nonumber
\ee
Here $(h^{-1}\dot h)^{p-1}$ and $(h^{-1}\dot h)^p$ are the induced endomorphisms on $\ti S^{p-1}$ and $\ti Q^p$. Thus we now have:
\bea
{\rm Tr}(\pi^*(\hat Fh^{-1}\dot h))&=&{\rm Tr}\,(\hat F^{\ti S^{p-1}}(h^{-1}\dot h)^{p-1}+ \pi^*\hat F^{\ti Q^p}(h^{-1}\dot h)^p)+\nonumber\\
&&\pi^*g^{j\bar k}\,{\rm Tr}( \gamma_{\bar k}\gamma^\dagger_j(h^{-1}\dot h)^{p-1}- \nabla_j\gamma_{\bar k}\dot V^\dagger-\nabla_{\bar k}\gamma^\dagger_j\dot V-\gamma^\dagger_j\gamma_{\bar k}(h^{-1}\dot h)^{p})\nonumber
\eea
We note that the term: 
\be
\int_0^1\int_{\ti X}{\rm Tr}(\pi^* \hat F^{\ti Q^p}(h^{-1}\dot h)^p)\pi^*\o^n,\nonumber
\ee
combines with:
\be
-\mu(Q^p)\int_{\ti X}({\rm log\,det}(h_1)-{\rm log\,det}(h_1^{p-1}))\pi^*\o^n,\nonumber
\ee
to give $M_p(H_0,H,\o)$. Also the term:
\be
\int_0^1\int_{\ti X}{\rm Tr}( \hat F^{\ti S^{p-1}}(h^{-1}\dot h)^{p-1})\pi^*\pi^n,\nonumber
\ee
combines with 
\be
-\sum_{i=1}^{p-1}\mu(Q^i)\int_0^1\int_X{\rm Tr}((\pi^i_t-\pi^{i-1}_t)h_t^{-1}\dot h_t)\,\pi^*\o^n\,dt,\nonumber
\ee
to give $P_{|_{\ti S^{p-1}}}(H_0,H)$. Thus the remaining term to identify is 
\be
\int_0^1\int_{\ti X}\pi^*g^{j\bar k}{\rm Tr}( \gamma_{\bar k}\gamma^\dagger_j(h^{-1}\dot h)^{p-1}- \nabla_j\gamma_{\bar k}\dot V^\dagger-\nabla_{\bar k}\gamma^\dagger_j\dot V-\gamma^\dagger_j\gamma_{\bar k}(h^{-1}\dot h)^{p})\pi^*\o^n\,dt.\nonumber
\ee
Now, since $\gamma_{\bar k}=\pl_{\bar k}p^\dagger$, we have $\dot \gamma_{\bar k}=\pl_t(\gamma_{\bar k}-(\gamma_{\bar k})_0)=\pl_t(\pl_{\bar k}(p^\dagger-p^\dagger_0))=\nabla_{\bar k}\dot V.$ Thus we can integrate by parts to get:
\be
\int_0^1\int_{\ti X}\pi^*g^{j\bar k}{\rm Tr}( \gamma_{\bar k}\gamma^\dagger_j(h^{-1}\dot h)^{p-1}+\gamma_{\bar k}\dot\gamma_j^\dagger+\gamma^\dagger_j\dot\gamma_{\bar k}-\gamma^\dagger_j\gamma_{\bar k}(h^{-1}\dot h)^{p})\pi^*\o^n\,dt.\nonumber
\ee
Consider the following formula, which can be found in $\cite{Don1}$:
\be
\pl_t(\gamma^\dagger_j)=\dot\gamma^\dagger_j+\gamma^\dagger_j(h^{-1}\dot h)^{p-1}-(h^{-1}\dot h)^p\gamma^\dagger_j.\nonumber
\ee
The final term now becomes:
\be
\int_0^1\int_{\ti X}\pl_t(\pi^*g^{j\bar k}{\rm Tr}(\gamma_{\bar k}\gamma^\dagger_j)\pi^*\o^n)\,dt=||\gamma^p||^2_{L^2}-||\gamma^p_0||^2_{L^2}.\nonumber
\ee
This completes the first step of the decomposition. We can continue the process on $P_{|_{\ti S^{p-1}}}(H_0,H)$ to prove the desired decomposition formula $\eqref{decomp}$. We are now ready to prove Theorem $\ref{lowerb}$.
\begin{proof}
 By $\eqref{decomp}$, we know the $P$-functional is the sum over all $i$ of three terms. The two second fundamental form terms are bounded below since $||\gamma^i||^2_{L^2}$ is positive and $-||\gamma_0^i||^2_{L^2}$ is fixed and only depends on our initial metric $H_0$. To see that $M_i(H_0,H,\o)$ is bounded below, notice that this functional is equivalent to the Donaldson functional defined on some regularization $\ti Q^i$. This regularization is a holomorphic vector bundle over $\ti X$, and is semi-stable with respect to the pulled back form $\pi^*\o$. Here we have expanded the definition of stability to include degenerate metrics (see Definitions $4$ and $5$ from \cite{J1} for details). This last term, the Donaldson functional defined on $\ti Q^i$, is explicitly shown to be bounded from below in the proof of Theorem 3 from $\cite{J1}$. 

\end{proof}

\section{An $L^2$ approximate Hermitian structure}

We are now ready to construct an $L^2$ approximate Hermitian structure on $E$ along the Yang-Mills flow, proving Theorem $\ref{firstmaintheorem}$. Recall that we defined the $P$-functional as the integral along a path, and proved this integral is path independent. Thus if $H_t$ is a family of metrics on $E$, the derivative in $t$ of the $P$ functional is readily seen as:  
\be
\pl_t P(H_0,H_t)=\int_X{\rm Tr}((\hat F-\Psi)H_t^{-1}\pl_t H_t)\o^n.\nonumber
\ee
We use the lower bound on the $P$-functional to show that $E$ admits an $L^2$ approximate Hermitian structure along the Donaldson heat flow. We then show the existence of such a structure along the Donaldson heat flow shows one exists along the Yang-Mills flow. 

First we need the following proposition, which gives one inequality in the proof of the Atiyah-Bott formula.
\begin{proposition}
{\label{AB1}}
For any connection $A\in{\cal A}^{1,1}$, we have:
\be
||\Psi||_{L^2}^2\leq ||\hat F_A||_{L^2}.\nonumber
\ee
\end{proposition}

The proof of this result is similar to the proof of Corollary 2.22 in $\cite{DW}$. We include the details here for the reader's convenience. As a first step, we compute the square of $\Psi$:
\be
\label{square}
\Psi^2=\sum_i\mu (Q^i)^2(\pi^i-\pi^{i-1}).
\ee
To see this, note for any $k>0$, we have $\pi^{i-k}\pi^i=\pi^i\pi^{i-k}=\pi^{i-k}$ since the subbundles are ordered by inclusion. Thus $(\pi^i-\pi^{i-1})^2={\pi^i}^2-\pi^{i-1}\pi^i-\pi^i\pi^{i-1}+{\pi^{i-1}}^2=\pi^i-\pi^{i-1}$. Also, all the cross terms in $\Psi_H^2$ vanish, since 
\bea
(\pi^{i-k}-\pi^{i-k-1})(\pi^i-\pi^{i-1})&=&(\pi^{i-k}(\pi^i-\pi^{i-1})-\pi^{i-k-1}(\pi^i-\pi^{i-1}))\nonumber\\
&=&(\pi^{i-k}-\pi^{i-k}-\pi^{i-k-1}+\pi^{i-k-1})\nonumber\\
&=&0.\nonumber
\eea
This proves $\eqref{square}$. 

Now, let $r$ be the rank of $E$. Recall the Harder-Narasimhan type of $E$:
\be
\vec{\mu}:=(\mu_1,...,\mu_i,...,\mu_r)=(\mu(Q^1),...,\mu(Q^1),\mu(Q^2),...,\mu(Q^2),...,\mu(Q^p),...,\mu(Q^p)),\nonumber
\ee
where the multiplicity of each $\mu(Q^i)$ is given by $rk(Q^i)$. Because $X$ has volume one, we see by $\eqref{square}$ that $||\Psi||_{L^2}^2=\sum_{i=1}^{r}\mu_i^2,$ which is independent of the metric used to define the projections $\pi^i$. Following Atiyah and Bott, given any two $r$-tuples $\vec{\mu},\vec{\lambda}$ satisfying $\mu_{i}\geq\mu_{i+1}$, $\lambda_i\geq\lambda_{i+1}$, and $\sum_{i=1}^{r}\mu_i=\sum_{i=1}^{r}\lambda_i$, we have
\be
\vec{\mu}\leq\vec{\lambda}\,\,\,\,{\Leftrightarrow}\,\,\,\,\sum_{j\leq k}^{r}\mu_j\leq\sum_{j\leq k}^{r}\lambda_i\,\,\,\,{\rm for\,all}\,\, k=1,...,r.\nonumber
\ee

Next we consider the convergence results of Hong-Tian. They show that on $X\backslash Z_{an}$ and along a subsequence, $A_j\longrightarrow A_\infty$ in $C^\infty$ modulo unitary gauge transformations. Since $|\hat F_t|_{C^0}$ is uniformly bounded along the flow (see Corollary 17 in \cite{Don1}), the $L^2$ norm  of $\hat F_\infty$ is defined on all of $X$, and we see that $||\hat F_j-\hat F_\infty||_{L^2(X)}^2$ goes to zero as $j$ tends to infinity. Furthermore, since $A_\infty$ is Yang-Mills, we have that the eigenvalues of $\hat F_\infty$ are locally constant, given by $\lambda_1\geq...\geq\lambda_r$ (counted with multiplicities). Consider the following lemma:

\begin{lemma}
Let $S$ be any torsion free subsheaf of $E$ of rank $s$. Then $deg(S)\leq\sum_{i\leq s}\lambda_i$.
\end{lemma}
\begin{proof}
Recall that along the Yang-Mills flow, the holomorphic structure of $E$ evolves by the action of $w_j$, where $w_j=h^{\frac{1}{2}}_j$ and $h_j=H_0^{-1}H_j$. Let $\pi_j$ be the orthogonal projection onto the subsheaf $w_j(S)$. We then have:
\bea
deg(S)&=&\int_X{\rm Tr}(\hat F_j\circ\pi_j)\o^n-||\bar\pl_j\pi_j||^2_{L^2}\nonumber\\
&\leq& \int_X{\rm Tr}(\hat F_\infty\circ\pi_j)\o^n+\int_X{\rm Tr}((\hat F_j-\hat F_\infty)\circ\pi_j)\o^n\nonumber\\
&\leq&\int_X{\rm Tr}(\hat F_\infty\circ\pi_j)\o^n+||\hat F_j-\hat F_\infty||_{L^2}.\nonumber
\eea
We need the following claim from linear algebra, which can be found in \cite{DW}:
\begin{claim}
Let $V$ be a finite dimensional Hermitian vector space of complex dimension $r$. Let $L\in End(V)$ be a Hermitian operator with eigenvalues $\lambda_1\geq...\geq\lambda_r$ (counted with multiplicities). Let $\pi$ denote orthogonal projection onto a subspace of dimension $s$. Then ${\rm Tr}(L\pi)\leq\sum_{i\leq s}\lambda_i$.
\end{claim}
Using this claim, we se that
\be
deg(S)\leq\sum_{i\leq s}\lambda_i+||\hat F_j-\hat F_\infty||_{L^2},\nonumber
\ee
and the result follows by sending $j$ to infinity. 
\end{proof}
Now, let $A_j$ be a sequence of connections along the Yang-Mills flow with Hong-Tian limit $A_\infty$. Let $\vec{\mu}$ be the Harder-Narasimhan type of $E$, and $\vec{\lambda}$ the eigenvalues of $\Lambda F_\infty$. Then we have that $\vec{\mu}\leq\vec{\lambda}$. To see this, recall $S^i$ are the subsheaves defining the Harder-Narasimhan filtration, and let them have corresponding rank $s^i$. By the previous lemma we have $deg(S^i)\leq\sum_{j\leq s^i}\lambda_j$ for all $i$. We also have that $deg(S^i)=\sum_{j\leq s^i}\mu_j$. Thus for each $s^i$ we have
\be
\sum_{j\leq s^i}\mu_j\leq\sum_{j\leq s^i}\lambda_j.\nonumber
\ee
Now, $\vec{\mu}\leq\vec{\lambda}$ follows from Lemma 2.3 in $\cite{DW}$. We are now ready to prove Proposition $\ref{AB1}$.
\begin{proof}
For any initial connection $A_0$ on $E$, let $A_t$ be the solution of the Yang-Mills flow. Since the Hermitian-Yang-Mills energy is decreasing along the flow, we know that for all $t$:
\be
||\hat F_{A_0}||^2_{L^2}\geq||\hat F_{A_t}||^2_{L^2}.\nonumber
\ee
By the convergence results of Hong-Tian there exists a subsequence so that $\hat F_j\longrightarrow \hat F_\infty$ in $L^2$, and $\hat F_\infty$ has constant eigenvalues $\vec{\lambda}$. We then have:
\be
||\hat F_{A_t}||^2_{L^2}\geq ||\hat F_\infty||_{L^2}^2=\sum_{i=1}^r\lambda_i^2.\nonumber
\ee
Yet as we have just seen, $\vec{\mu}\leq\vec{\lambda}$. Thus from Proposition 12.6 in $\cite{AB}$ we have
\be
\sum_{i=1}^r\lambda_i^2\geq \sum_{i=1}^r\mu_i^2.\nonumber
\ee
The proof of the proposition is complete by noting $\sum_{i=1}^r\mu_i^2=||\Psi||^2_{L^2},$ so
\be
||\hat F_{A_0}||^2_{L^2}\geq||\hat F_{A_t}||^2_{L^2}\geq ||\hat F_\infty||_{L^2}^2=\sum_{i=1}^r\lambda_i^2\geq\sum_{i=1}^r\mu_i^2=||\Psi||^2_{L^2}.\nonumber
\ee
\end{proof}

We now turn to the main proposition in this section. Before we begin, we first introduce two differential operators used in the proof. Recall that $\nabla$ denotes the Chern connection on all associated bundles of $E$. We define the following Laplacians, which we write down in local coordinates
\be
\Delta=g^{j\bar k}\nabla_j\nabla_{\bar k}\qquad{\rm and}\qquad \bar\Delta=g^{j\bar k}\nabla_{\bar k}\nabla_j.\nonumber
\ee
We remark these operators are defined using the ``analyst convention.'' Now, from equation (2.1.1) from \cite{Siu}, one sees that along the Donaldson heat flow, the evolution of the curvature endomorphism is given by
\be
\pl_t(\hat F)=-g^{j\bar k}\nabla_{\bar k}\nabla_j(H^{-1}\pl_t H)=\bar\Delta( \hat F).\nonumber
\ee
In this special case, we can replace $\bar\Delta$ with $\Delta$, as the difference is given by a commutator $[ \hat F, \hat F]$, which vanishes. Thus $\pl_t(\hat F)=\Delta( \hat F)=\bar\Delta( \hat F)$, and this important fact will be utilized in the following proposition.

\begin{proposition}
\label{DHFAPS}
Along the Donaldson heat flow we have the following convergence:
\be
||\hat F_t-\Psi_t||_{L^2}^2\longrightarrow 0,\nonumber
\ee
as $t$ approaches infinity. 
\end{proposition}
\begin{proof}
We first show that we have $L^2$ convergence of $\hat F_t-\Psi_t$ along a subsequence. 
Note that this $L^2$ norm is computed with respect to the evolving metric $H$. As a first step, we show that for all times along the Donaldson heat flow, we have the inequality:
\be
\label{keyineq}
||\hat F_t-\Psi_t||_{L^2}^2\leq 2 \int_X{\rm Tr}((\hat F_t-\Psi_t)(\hat F_t-\mu(E)I))\o^n.
\ee
Expanding out the left hand side gives:
\be
\int_X{\rm Tr}(\hat F_t^2-2\hat F_t\Psi_t+\Psi_t^2)\o^n,\nonumber
\ee
while the right hand side is given by
\be
\int_X{\rm Tr}(2\hat F_t^2-2\hat F_t\Psi_t-2\mu(E)\hat F_t+2\mu(E)\Psi_t)\o^n.\nonumber
\ee
Thus the inequality $\eqref{keyineq}$ reduces to
\be
||\Psi||^2_{L^2}\leq||\hat F_t||^2_{L^2}+2\mu(E)\int_X{\rm Tr}({\Psi_t}-\hat F_t)\o^n.\nonumber
\ee
Yet $\Psi_t$ is constructed so that Tr$(\Psi_t)=deg(E)$, so the second term on the right vanishes. Thus inequality $\eqref{keyineq}$ follows from Proposition $\ref{AB1}$. Note that the term on the right hand side of inequality $\eqref{keyineq}$ is minus the time derivative of the $P$ functional along the Donaldson heat flow.

We now have that:
\bea
\int_0^\infty||\hat F_t-\Psi_t||_{L^2}^2dt&\leq& 2 \int_0^{\infty}\int_X{\rm Tr}((\hat F_t-\Psi_t)(\hat F_t-\mu(E)I))\o^ndt\nonumber\\
&=&-2\int_0^{\infty}\pl_t P(H_0,H(t))dt\nonumber\\
&=&2 P(H_0,H_0)-\lim_{t\rightarrow\infty} 2 P(H_0,H(t))\leq C,\nonumber
\eea
which is bounded since the $P$ functional is bounded from below. Thus there exists a sequence of times $t_i$ along the Donaldson heat flow such that
\be
Y(t_i):=||\hat F_{t_i}-\Psi_{t_i}||_{L^2}^2\longrightarrow 0,\nonumber
\ee
proving there exists an $L^2$ approximate Hermitian structure on $E$. Note that
\be
\sum_{m=0}^\infty\int_m^{m+1}Y(t)dt<\infty,\nonumber
\ee
hence we can find a sequence $t_m$ such that $t_m\in[m,m+1)$ and $Y(t_m)\longrightarrow 0$.

Next we show such a structure exists for any subsequence of times. Note that $\Psi$ can be rewritten as:
\be
\Psi=\sum_i c_i\pi^i,\nonumber
\ee
where the $c_i$ are all positive constants since the Harder-Narasimhan filtration is slope decreasing. First we see that the time derivative of the projection $\pi^i$ along the path $H_t$ is given by
\be
\label{evolvepi}
\pl_t\pi^i=\pi^i(H_t^{-1}\pl_t H_t)(I-\pi^i),
\ee
 as long as $H_t^{-1}\pl_t H_t$ is self adjoint with respect to $H$. To see this, first note that the image of $\pl_t\pi^i$ lies in $S$ since it is given by the difference of two projections. Consider two sections $\phi, \psi\in\Gamma(X,E)$ such that $\phi$ lies in the image of $f_i$. Then one has
 \be
 0=\langle \phi, (I-\pi^i) \psi\rangle_H.\nonumber
 \ee
 Taking the time derivative of the above expression we see 
 \be
 0=-\langle \phi, \pl_t\pi^i \psi\rangle_H+\langle H^{-1}\pl_tH\phi, (I-\pi^i) \psi\rangle_H,\nonumber
 \ee
 where the second term on the right comes from the time derivative falling on the metric $H$. From here \eqref{evolvepi} follows. 
 We compute how $Y(t)$ evolves with time:
\bea
\pl_tY(t)&=&2\int_X{\rm Tr}((\pl_t\hat F-\pl_t\Psi)(\hat F-\Psi))\o^n\nonumber\\
&=&\int_X{\rm Tr}(\Delta \hat F(\hat F-\Psi))\o^n-\sum_ic_i\int_X{\rm Tr}(\pi^i(H_t^{-1}\pl_t H_t)(I-\pi^i)(\hat F-\Psi))\o^n\nonumber\\
&\leq&\int_Xg^{j\bar k}{\rm Tr}(\nabla_{\bar k}\hat F\nabla_j\Psi)\o^n+\sum_ic_i\int_X{\rm Tr}(\pi^i(\hat F-\mu(E)I)(I-\pi^i)(\hat F-\Psi))\o^n,\nonumber
\eea
after integration by parts on the first term. Note that $\pi^i \mu(E) I (I-\pi^i)=0$. Now, the part of $\hat F$ that sends ${S^i}^\perp$ to $S^i$ is $-g^{j\bar k}\ti \nabla_j\nabla_{\bar k}\pi^i$, where $\ti\nabla_j$ is the covariant derivative for $Hom({S^i}^\perp,S^i)$. The difference between $\ti\nabla_j$ and $\nabla_j$ is two second fundamental form terms
\be
-g^{j\bar k}\ti \nabla_j\nabla_{\bar k}\pi^i=-g^{j\bar k}\nabla_j\nabla_{\bar k}\pi^i+g^{j\bar k}\nabla_j\pi^i\nabla_{\bar k}\pi^i-g^{j\bar k}\nabla_{\bar k}\pi^i\nabla_j\pi^i.\nonumber
\ee
Thus it follows that
\bea
\int_X{\rm Tr}(\pi^i(\hat F-\mu(E)I)(I-\pi^i)(\hat F-\Psi))\o^n=-\int_X{\rm Tr}(g^{j\bar k}\ti \nabla_j\nabla_{\bar k}\pi^i)(\hat F-\Psi))\o^n,\nonumber
\eea
and the right hand side is bounded by
\be
-\int_X{\rm Tr}(g^{j\bar k}\nabla_j\nabla_{\bar k}\pi^i(\hat F-\Psi))\o^n+C||\nabla\pi^i||^2_{L^2},\nonumber
\ee
where we used the $L^\infty$ bound for $\hat F$ (again by Corollary 17 from \cite{Don1}). We then have:
\be
\sum_ic_i\int_X{\rm Tr}(g^{j\bar k}\nabla_j\nabla_{\bar k}\pi^i\Psi)\o^n=\int_X{\rm Tr}(g^{j\bar k}\nabla_j\nabla_{\bar k}\Psi\Psi)=-||\nabla(\Psi)||^2_{L^2}\leq0.\nonumber
\ee
Thus returning to our initial computation of $\pl_t Y(t)$ and applying H\"older's inequality we have:
\bea
\pl_tY(t)&\leq& C_0||\nabla \hat F||_{L^2}||\nabla \Psi||_{L^2}+C||\nabla\Psi||^2_{L^2}\nonumber\\
\label{finalone}
&\leq&C_0||\nabla \hat F||^2_{L^2}+C||\nabla\Psi||^2_{L^2}.
\eea
We show both terms on the right hand side go to zero as $t$ approaches infinity. Set $f(t)=||\nabla \hat F||_{L^2}^2$. Then Proposition 9 from $\cite{HT}$ shows precisely that  $f(t)\longrightarrow 0$ as $t$ goes to infinity.

We now concentrate on $||\nabla\Psi||^2_{L^2}$. As a first step we show that for any $\pi^i$ from $\Psi$ we have that $||\nabla\pi^i||^2_{L^2}$ goes to zero along a subsequence, and to do so we need a modification of the Chern-Weil formula. Once again recall that $\o$ is normalized so $\int_X\o^n=1$. We have
\be
\int_X{\rm Tr}(\Psi\circ\pi^i)\o^n=\sum_k\int_X{\rm Tr}(\mu(Q^k)(\pi^k-\pi^{k-1})\circ\pi^i)\o^n.\nonumber
\ee
However, if $k\geq i$, then because the Harder-Narasimhan filtration is ordered by inclusion we know $\pi^k\circ\pi^i=\pi^i$, so

\bea
\int_X{\rm Tr}(\Psi\circ\pi^i)\o^n&=&\sum_{k\leq i}\int_X{\rm Tr}(\mu(Q^k)(\pi^k-\pi^{k-1}))\o^n.\nonumber\\
&=&\sum_{k\leq i}\mu(Q^k)\,rk(Q^k)=\sum_{k\leq i}{\rm deg}(Q^k).\nonumber
\eea
We note that ${\rm deg}(Q^k)={\rm deg}(S^k)-{\rm deg}(S^{k-1})$, so the sum $\sum_{k\leq i}{\rm deg}(Q^k)$ is a telescoping sum. Thus the only contribution is the term coming from $k=i$, so by the Chern-Weil formula:
\be
\label{CW}
\sum_{k\leq i}{\rm deg}(Q^k)={\rm deg}(S^i)=\int_X{\rm Tr}(\hat F\circ\pi^i)-||\nabla\pi^i||^2_{L^2}.\nonumber
\ee
Thus
\be
\int_X{\rm Tr}(\Psi\circ\pi^i)\o^n=\sum_{k\leq i}{\rm deg}(Q^k)=\int_X{\rm Tr}(\hat F\circ\pi^i)-||\nabla\pi^i||^2_{L^2}.\nonumber
\ee
Therefore, for each projection $\pi^{i}_j$ in our sequence along the Yang-Mills flow, we have the following formula:
\be
||\nabla\pi^i||^2_{L^2}=\int_X{\rm Tr}((\hat F-\Psi)\circ\pi^{i})\o^n.\nonumber
\ee
And because the eigenvalues of $\pi^{i}$ are either $0$ or $1$, it follows that
\be
\label{sffb}
||\nabla \pi^{i}||^2_{L^2}\leq\int_X|\hat F-\Psi|\o^n.
\ee
This gives:
 \be
 ||\nabla\pi^i||^2_{L^2}\leq||\hat F-\Psi||_{L^1}\leq||\hat F-\Psi||_{L^2},\nonumber
\ee
where the last inequality follows since the volume of $X$ is one. Squaring both sides gives $||\nabla\pi^i||^4_{L^2}\leq Y(t)$, and thus $||\nabla\pi^i(t_m)||^4_{L^2}$ goes to zero as $m$ goes to infinity. Therefore the square root $||\nabla\pi^i(t_m)||^2_{L^2}$ goes to zero as $m$ goes to infinity as well. 

We now show $||\nabla\pi^i||^2_{L^2}$ goes to zero for all time $t$ approaching infinity. To do so we prove a simple differential inequality. We apply the Chern-Weil formula to the projection evolving along the flow, and take the derivative in time:
\bea
\pl_t||\nabla\pi^i||^2_{L^2}&=&\int_X{\rm Tr}(\pl_t\hat F \pi^i)\o^{n}+\int_X{\rm Tr}(\hat F\pl_t\pi^i)\o^{n}\nonumber\\
&=&\int_X{\rm Tr}(\bar\Delta \hat F \pi^i)\o^{n}-\int_X{\rm Tr}(\hat F\pi^i\hat F(I-\pi^i))\o^{n},\nonumber
\eea
 Note the second integral on the right is non-positive. Now we integrate by parts twice:
\bea
\pl_t||\nabla\pi^i||^2_{L^2}&\leq&\int_X{\rm Tr}(\hat F g^{j\bar k}\nabla_j\nabla_{\bar k}\pi^i)\o^{n}\nonumber\\
&=&\int_X{\rm Tr}(\hat F (g^{j\bar k}\ti\nabla_j\nabla_{\bar k}\pi^i+g^{j\bar k}\nabla_j\pi^i\nabla_{\bar k}\pi^i-g^{j\bar k} \nabla_{\bar k}\pi^i\nabla_j\pi^i))\o^{n}\nonumber\\
&\leq&-\int_X{\rm Tr}(\hat F\pi^i\hat F(I-\pi^i))\o^{n}+C||\nabla\pi^i||^2_{L^2}\leq C||\nabla\pi^i||^2_{L^2}.\nonumber
\eea
This estimate, along with the fact proven in the last paragraph that $||\nabla\pi^i||^2_{L^2}$ approaches zero along a subsequence $t_m\in[m,m+1)$, implies that $||\nabla\pi^i||^2_{L^2}$ goes to zero for all time $t$ approaching infinity (for details see \cite{PSSW}).

We now return to $\eqref{finalone}$. Set $g(t)=||\nabla \Psi||^2_{L^2}$. We have seen that both $f(t)$ and $g(t)$ go to zero for all $t$ approaching infinity. Pick a $t\in[m+1,m+2)$. Integrating both sides of  $\eqref{finalone}$ from $t_m$ to $t$ we get
\bea
\int_{t_m}^t\pl_sY(s)ds&\leq&C\int_{t_m}^tf(s)+g(s)ds,\nonumber
\eea
which implies
\be
Y(t)\leq Y(t_m)+2C\sup_{s\in(m,m+2)}(f(s)+g(s)).\nonumber
\ee
Sending $m$ to infinity we see $Y(t)$ goes to zero. Thus there exists an $L^2$ approximate Hermitian structure along the Donaldson heat flow. 
\end{proof}

At this point we can now prove Theorem $\ref{lastcor}$ as stated in the introduction, generalizing a result of Atiyah and Bott. 
First we review some notation. Consider a flag $\cal F$ of subbundles:
\be
0=E^0\subset E^1\subset\cdots\subset E^q=E.\nonumber
\ee
Define $\cal F$ to be slope decreasing if $\mu(E^1)>\mu(E^2)>...>\mu(E)$. Let ${\cal Q}^i=E^i/E^{i-1}$, and recall that 
\be
\Phi({\cal F})^2=\sum_{i=0}^q\mu({\cal Q}^i)^2rk({\cal Q}^i).\nonumber
\ee
We now prove that for all $\cal F$ slope decreasing:
\be
\inf_A||\hat F_A||^2_{L^2}=\sup_{\cal F}\Phi({\cal F})^2.\nonumber
\ee
\begin{proof}

We begin by showing $\sup_{\cal F}\Phi({\cal F})^2=||\Psi_H||^2_{L^2}$. Since we already know the supremum is attained if $\cal F$ is the Harder-Narasimhan filtration of $E$, all we need is  $||\Psi_H||^2_{L^2}=\sum_{i=0}^p\mu(Q^i)^2rk(Q^i),$ which follows directly from $\eqref{square}$. We drop the $H$ from $\Psi_H$ since this norm is independent of metric. To conclude we show $\inf_A||\hat F_A||^2_{L^2}=||\Psi||^2_{L^2}$. One inequality follows from Proposition $\ref{AB1}$, which states $||\hat F_H||^2_{L^2}\geq||\Psi||^2_{L^2}$ for all metrics $H$. The infimum is now obtained by taking a sequence of metrics $H_t$ where $t\longrightarrow\infty$ along the Donaldson heat flow.
\end{proof}

We now demonstrate how Proposition $\ref{DHFAPS}$ gives an $L^2$ approximate Hermitian structure along the Yang-Mills flow. First we state a fact about adjoints. In local coordinates, the adjoint of an endomorphism $T$ with respect to the metric $H$ is given by
\be
T^*{}^\al{}_\b=H^{\al\bar\gamma}\overline{ T^\rho{}_\gamma}H_{\bar \rho\b}.\nonumber
\ee 
Notice now that if we wanted to compute the adjoint with respect to the metric $H_0$, denoted by $*_0$, we have
\be
T^{*_0}{}^\al{}_\b=H_0^{\al\bar\gamma}\overline{ T^\rho{}_\gamma}H_{0\bar \rho\b}=H_0^{\al\bar\gamma}H_{\bar\gamma\nu}H^{\nu\bar\kappa}\overline{ T^\sigma{}_\kappa}H_{\bar \sigma\eta}H^{\eta\bar\rho}H_{0\bar \rho\b}=h^\al{}_\nu T^*{}^\nu{}_\eta h^{-1}{}^\eta{}_\b.\nonumber
\ee
Thus in matrix notation we have $T^{*_0}=hT^*h^{-1}$. We now see that $||\hat F||_{L^2(H)}^2=||\hat F_A||^2_{L^2(H_0)}$, since by $\eqref{GRL}$ we have
\bea
||\hat F||_{L^2(H)}^2&=&\int_X{\rm Tr}(\hat F\hat F^*)\o^n=\int_X{\rm Tr}(w^{-1}\hat F_Aw(w^{-1}\hat F_Aw)^*)\o^n\nonumber\\
&=&\int_X{\rm Tr}(\hat F_Ah\hat F_A^*h^{-1})\o^n=\int_X{\rm Tr}(\hat F_A\hat F_A^{*_0})\o^n=||\hat F_A||^2_{L^2(H_0)}.\nonumber
\eea

Next we need to relate our projections evolving along the Donaldson heat flow to projections evolving along the Yang-Mills flow. In the case of the Donaldson heat flow, the orthogonal projection $\pi_t$ onto a fixed subsheaf $S\subset E$ evolves due to the fact that the metric $H$ is changing. Along the Yang-Mills flow, our metric $H_0$ is fixed, however the subsheaf $S$ is acted on by the complexified gauge transformation $w$. Thus the projection $\pi_w$ onto $w(S)$ evolves as well. 
\begin{lemma}
The two evolving projections are related as follows
\be
\pi_w=w\pi_tw^{-1}\nonumber
\ee
\end{lemma}
\begin{proof}
It is immediately clear that  $(w\pi_tw^{-1})^2=w\pi_tw^{-1}$, so $w\pi_tw^{-1}$ is a projection onto the subsheaf $w(S)$. We complete the lemma by showing it is self-adjoint with respect to $H_0$. 
\be
(w\pi_tw^{-1})^{*_0}=w^{-1}(\pi_t)^{*_0}w=h^{-1/2}h\pi_t^*h^{-1}h^{1/2}=w\pi_tw^{-1}.\nonumber
\ee
\end{proof}
From this lemma we see that $w\Psi_tw^{-1}=\Psi_w$, where $\Psi_t$ is evolving along the Donaldson heat flow and $\Psi_w$ is evolving along the Yang-Mills flow. It follows that
\bea
||\hat F_t-\Psi_t||^2_{L^2(H)}&=&\int_X{\rm Tr}((w^{-1}\hat F_Aw-\Psi_t)(w^{-1}\hat F_Aw-\Psi_t)^*)\o^n\nonumber\\
&=&\int_X{\rm Tr}(\hat F_A-w\Psi_tw^{-1})h(\hat F_A-w\Psi_tw^{-1})^*h^{-1})\o^n\nonumber\\
&=&\int_X{\rm Tr}(\hat F_A-\Psi_w)(\hat F_A-\Psi_w)^{*_0})\o^n\nonumber\\
&=&||\hat F_A-\Psi_w||^2_{L^2(H_0)}.\nonumber
\eea
Thus we see that the Yang-Mills flow realizes an $L^2$ approximate Hermitian structure as well, proving Theorem $\ref{firstmaintheorem}$. From this point on we abuse notation and also refer to the endomorphism evolving along the Yang-Mills flow as $\Psi_t$, and which endomorphism we are using will be clear from context.

\section{Construction of an isomorphism}

In this section we use Theorem $\ref{firstmaintheorem}$ to better understand our limiting connection $A_{\infty}$, and following our previous work \cite{J1, J2} we are able to prove Theorem $\ref{main theorem}$. We recall our basic setup. Let $A_t$ be a family of connections evolving along the Yang-Mills flow. By a result of Hong and Tian $\cite{HT}$, there exists a subsequence of connections $A_j$ which converge in $C^\infty$ (on $X\backslash Z_{an}$ and modulo unitary gauge transformations), to a Yang-Mills connection $A_\infty$. Thus, always working on $X\backslash Z_{an}$,  we have a sequence of holomorphic structures $(E,\bar\pl_j)$ which converge in $C^\infty$ to a holomorphic structure $(E,\bar\pl_\infty)$. We denote this limiting holomorphic bundle by $E_\infty$.

We can now identity the Harder-Narasimhan type of the limiting connection $A_\infty$. 
Since $A_\infty$ is Yang-Mills, we have that $\hat F_\infty$ solves the following equation:
\be
-i\bar\pl_\infty\Lambda F_\infty+i\pl_\infty\Lambda F_\infty=0.\nonumber
\ee 
In particular $\hat F_\infty$ has locally constant eigenvalues. Because $\hat F_\infty$ is Hermitian, about any point in $X\backslash Z_{an}$, we can choose coordinates so that $\hat F_\infty$ has the following form:
\be
\label{fmatrix}
\hat F_\infty=\left( 
\begin{array}{cccc}
\lambda_1 I_1 & 0 & \cdots & 0\\
0 & \lambda_2 I_2  &\cdots&0\\
\vdots &\vdots& \ddots&\vdots\\
0&0&\cdots&\lambda_p I_p\end{array} 
\right).
\ee
Here $I_i$ are identity matrices whose rank is determined by the multiplicity of each eigenvalue $\lambda_i$. Assume that the eigenvalues are decreasing $\lambda_1>\lambda_2>\cdots>\lambda_q$. Now because $E$ realizes an $L^2$ approximate Hermitian structure along the Yang-Mills flow, we can precisely identify the eigenvalues of $\hat F_\infty$, so $\lambda_i=\mu(Q^i)$, and $rk(I_i)=rk(Q^i)$.

Furthermore, because $\hat F_\infty$ is of this special form, we know it will decompose $E_\infty$ into a direct sum of stable bundles:
\be
\label{decomp2}
E_\infty=\hat Q^1_\infty\oplus \hat Q^2_\infty\oplus\cdots \oplus \hat Q^q_\infty,
\ee
each admitting an induced smooth Hermitian-Einstein connection. Let $Z=Z_{an}\cup Z_{alg}$. Working on $X\backslash Z$, we prove the direct sum $\eqref{decomp2}$ is isomorphic to the graded double filtration $Gr^{hns}(E)$, which is the subject of the following proposition:

\begin{proposition}
\label{isom1}
Working with $\eqref{decomp2}$ above,  on $X\backslash Z$ each $\hat Q^i_\infty$ is isomorphic to a specific stable quotient from $Gr^{hns}(E)$.
\end{proposition}
We prove this proposition at the end of this section. First we need some convergence results. Consider the $L^2_1$ projections which define the Harder-Narasimhan filtration of $E$:
\be
\label{HNFp}
0\subset \pi^0\subset\pi^1\subset\pi^2\cdots\subset\pi^p\subset E.
\ee
Recall that along any one parameter family of connections we have a sequence of endomorphisms $w_j$ which define the action given by $\eqref{action}$. The action of $w_j$ also produces a sequence of filtrations $\{\pi^{i}_j\}$, where each $\pi^{i}_j$ is defined by orthogonal projection onto the subsheaf $w_j(\pi^{i})$. Our first goal is to show that this sequence of filtrations converges along a subsequence, with two assumptions on our sequence of connections. After the proposition we show these assumptions hold in our case.
\begin{proposition}
\label{limit}
Let $\pi$ be the $L^2_1$ projection associated to a subsheaf ${\cal F}\subset E$, and let $Z({\cal F})$ be the singular set of ${\cal F}$. Let $\{A_j\}$ be a sequence of connections, along with corresponding complexified gauge transformations $\{w_j\}$. The action of $w_j$ produces a sequence of projections $\{\pi_j\}$ defined by orthogonal projection onto the subsheaf $w_j(\pi)$. Assume that:

\medskip

i) For any compact subset $K\subset X\backslash(Z_{an}\cup Z({\cal F}))$, we have $A_j\longrightarrow A_\infty$ in $C^\infty(K)$.

\medskip

ii) $||\bar\pl_j\pi_j||^2_{L^2}\longrightarrow 0$.

\medskip
Then there exists a subsequence of projections (still denoted $\pi_j$) which converges in $L^2_1$ to a limiting subsheaf $\pi_\infty$. Furthermore, under the same assumptions the limiting projection $\pi_\infty$ is smooth away from $Z_{an}\cup Z({\cal F})$.
\end{proposition}
We note that assumption $ii)$ gives that $\pi_\infty$ splits $E_\infty$ holomorphically. A similar proposition is proved in $\cite{J1}$, however we include all the details here for the reader's convenience.
\begin{proof}

By assumption $ii)$ we have $||\bar\pl_j\pi_j||_{L^2}^2$ goes to zero as $j\rightarrow\infty$. Because $\pi_j=\pi_j^*$ it follows that $|\bar\pl_j\pi_j|^2=|\pl_j\pi_j|^2$, thus we have $\pl_j\pi_j$ is uniformly bounded in $L^2$ and $\pi_j$ converges along a subsequence to a weak limit $\pi_\infty$ in $L^2_1$. We must show that $\pi_\infty$ is a weakly holomorphic subbundle as defined in $\cite{Simp}$ or $\cite{UY}$, and thus represents a coherent subsheaf. This means we have to show $(I-\pi_\infty)\bar\pl_{\infty}\pi_\infty=0$ in $L^2$. Working on a compact set $K$ specified in assumption $i)$, we have: 
\be
\bar\pl_\infty\pi_j=\bar\pl_j\pi_j+(\bar\pl_\infty-\bar\pl_j)\pi_j\nonumber,
\ee
so it follows that
\bea
||\bar\pl_\infty\pi_j||_{L^2(K)}&\leq&||\bar\pl_j\pi_j||_{L^2(K)}+||(\bar\pl_\infty-\bar\pl_j)\pi_j||_{L^2(K)}\nonumber\\
&\leq&||\bar\pl_j\pi_j||_{L^2(K)}+||A_j-A_\infty||_{L^\infty(K)}||\pi_j||_{L^2(K)},\nonumber
\eea
We have that $A_j\rightarrow A_\infty$ in $L^\infty(K)$ by assumption $i)$. Because $||\bar\pl_j\pi_j||_{L^2}\rightarrow 0$ it follows that $||\bar\pl_\infty\pi_j||_{L^2(K)}\rightarrow 0$. Finally, from the simple formula:
\be
\bar\pl_\infty\pi_\infty=\bar\pl_\infty\pi_j+\bar\pl_\infty(\pi_\infty-\pi_j),\nonumber
\ee
we see that
\bea
||\bar\pl_\infty\pi_\infty||_{L^2(K)}&\leq&||\bar\pl_\infty\pi_j||_{L^2(K)}+||\bar\pl_\infty(\pi_\infty-\pi_j)||_{L^2(K)}\nonumber\\
&=&||\bar\pl_\infty\pi_j||_{L^2(K)}+||\pi_\infty-\pi_j||_{L^2_1(K)}.\nonumber
\eea
The left hand side is independent of $j$, so we would like to send $j$ to infinity proving $||\bar\pl_\infty\pi_\infty||_{L^2(K)}=0$. We have to be careful about the second term on the right since $\pi_j$ only converges to $\pi_\infty$ weakly in $L^2_1(K)$. However, we can achieve strong $L^2_1(K)$ convergence along a subsequence, as will now be demonstrated. Equation $\eqref{connectiondecomp}$ describes how a connection decomposes onto subbundles $\pi_j$ with quotient $Q_j$. From this formula we see that the second fundamental form is just one component of the connection $A_j$, so we have:
\be
\int_K|\ti\nabla_j (\bar\pl_j\pi_j)|^2\o^n\leq\int_K|\nabla_j(A_j)|^2\o^n\leq C,\nonumber
\ee
where $\ti\nabla_j$ is the induced connection on $Hom(Q_j,S_j)$. The bound on the right follows from assumption $i)$. Thus $\pi_j$ is bounded in $L^2_2(K)$, and along a subsequence we have strong convergence in $L^2_1(K)$. It follows that $||\bar\pl_\infty\pi_\infty||_{L^2(K)}=0$. This holds independent of which compact set $K$ we choose, so 
\be
||\bar\pl_\infty\pi_\infty||_{L^2(X\backslash Z_{an})}=||\bar\pl_\infty\pi_\infty||_{L^2(X)}=0,\nonumber
\ee
since $Z_{an}\cup Z({\cal F})$ has complex codimension at least two. Thus $\pi_\infty$ defines a weakly holomorphic $L^2_1$ 
subbundle of $(E_\infty,\bar\pl_\infty)$. Furthermore, because the eigenvalues of the projections $\pi_j$ are either zero or one, we know that rk$(\pi_\infty)$=rk$(\pi_j)$. It also follows that $\mu(\pi)=\mu(\pi_\infty)$, since degree does not depend on a choice of metric.

We now prove $\pi_\infty$ is smooth away from $Z_{an}\cup Z({\cal F})$. Fix an arbitrary compact subset $K\subset(Z_{an}\cup Z({\cal F}))$. By $\eqref{connectiondecomp}$ we see that the second fundamental form  $\gamma_j$ is one component of the decomposition of the connection $A_j$. On $K$ we have smooth convergence of $A_j$ to $A_\infty$ by assumption $i)$, so as a result we know the associated second fundamental forms $\gamma_j$ must converge smoothly to $\gamma_\infty$ as well. Thus along our subsequence, for any $k\in\Z$, we have:
\be
||\gamma_j-\gamma_\infty||_{C^k(K)}\longrightarrow 0.\nonumber
\ee
This smooth convergence of second fundamental forms, in addition to the fact that $\gamma_j=
\bar\pl_j\pi_j$, proves smooth convergence of the projections $\pi_j$. 

\end{proof}

We now show that the assumptions of Proposition $\ref{limit}$ hold for the projections that make up the Harder-Narasimhan filtration $\eqref{HNF}$ along the Yang-Mills flow. The convergence results of Hong and Tian $\cite{HT}$ (also, see $\cite{U}$) imply there exists a subsequence along the Yang-Mills flow that satisfies assumption $i)$. For assumption $ii)$, recall inequality $\eqref{sffb}$:
\be
||\bar\pl_j\pi^{i}_j||^2_{L^2}\leq\int_X|\hat F_j-\Psi_j|\o^n.\nonumber
\ee
$E$ admits an $L^2$ approximate Hermitian structure along the Yang-Mills flow, thus assumption $ii)$ holds for all subsheaves in $\eqref{HNF}$. We therefore get convergence to a limiting filtration away from $Z$:
\be
\pi^1_\infty\subset\cdots\subset\pi^p_\infty=E_\infty.\nonumber
\ee
In the following lemma we prove two important facts about the quotients $Q^i_\infty=\pi^i_\infty/\pi^{i-1}_\infty$:
\begin{lemma}
Each quotient  $Q^i_\infty=\pi^i_\infty/\pi^{i-1}_\infty$ is semi-stable. Furthermore, $E_\infty$ splits as a direct sum:
\be
\label{decomp1}
E_\infty=Q^1_\infty\oplus\cdots\oplus Q^{p-1}_\infty.
\ee
\end{lemma}

\begin{proof}
We begin with the subsheaf of highest rank $\pi^{p-1}_\infty$. Because the second fundamental form $||\bar\pl\pi^{p-1}_\infty||^2_{L^2}=0$, the induced curvature on $Q^p_\infty$ is just
\be
\hat F^{Q^p}_\infty=(I-\pi^{p-1}_\infty)\circ \hat F_\infty\circ(I-\pi^{p-1}_\infty).\nonumber
\ee
Thus, because $rk(I_{p})=rk(I-\pi^{p-1}_\infty)$, we know $\hat F^{Q^p}_\infty=\lambda_p I_p$ (where the eigenvalue $\lambda_p$ is defined in $\eqref{fmatrix}$). So $Q^p_\infty$  admits a Hermitian-Einstein connection. By the removable singularity theorem of Bando-Siu from $\cite{BaS}$,  $Q^p_\infty$ extends to a reflexive sheaf on all of $X$. Because it admits a Hermitian-Einstein connection where it is locally free it is semi-stable. 

Now, $||\bar\pl\pi^{p-1}_\infty||^2_{L^2}=0$ implies that $E_\infty$ splits as a direct sum $E_\infty=\pi^{p-1}_\infty\oplus Q^p_\infty$.
This splitting, along with the fact that $||\bar\pl\pi^{p-2}_\infty||^2_{L^2}=0$, implies the second fundamental form with respect to the inclusion $\pi^{p-2}_\infty\subset\pi^{p-1}_\infty$ is zero, from which it follows that:
\be
\hat F^{Q^{p-1}}_\infty=(\pi^{p-1}_\infty-\pi^{p-2}_\infty)\circ \hat F_\infty\circ(\pi^{p-1}_\infty-\pi^{p-2}_\infty).\nonumber
\ee
We continue in this way down the entire filtration. Each $Q^i_\infty$  admits a Hermitian-Einstein connection, and thus it is semi-stable. The decomposition $\eqref{decomp1}$ follows as well.

\end{proof}
Because each $Q^i_\infty$ is semi-stable admitting a Hermitian-Einstein connection, we know $Q^i_\infty$ will decompose into a direct sum of stable bundles. These stable bundles make up the direct sum $\eqref{decomp2}$, and it is on this level that we must construct the isomorphism with $Gr^{hns}(E)$.

\begin{lemma}
\label{AHES}
Given a sequence of connections $A_j$ along the Yang-Mills flow, the induced connections on $Q^i$ realize an $L^1$ approximate Hermitian-Einstein structure.
\end{lemma}
\begin{proof}
For a subbundle $\pi^i$ in the Harder-Narasimhan filtration, the induced curvature satisfies the following inequality:
\be
\int_X|\hat F^{S^i}|\o^n\leq\int_X|\pi^i\circ \hat F\circ \pi^i|\o^n+||\bar\pl\pi^i||^2_{L^2}.\nonumber
\ee
Now, because the second fundamental form for the inclusion $\pi^{i-1}\subset\pi^i$ is given by $\bar\pl\pi^{i-1}-\bar\pl\pi^i$, the induced curvature on $Q^i=S^i/S^{i-1}$ satisfies the following:
\bea
\int_X|\hat F^{Q^i}|\o^n&\leq&\int_X|(I-\pi^{i-1})\circ \hat F^{S^i}\circ (I-\pi^{i-1})|\o^n+||\bar\pl\pi^i||^2_{L^2}\nonumber\\
&&+||\bar\pl\pi^{i-1}||^2_{L^2}+2||\bar\pl\pi^i||_{L^2}||\bar\pl\pi^{i-1}||_{L^2}.\nonumber
\eea
Putting the last two inequalities together we see:
\bea
\int_X|\hat F^{Q^i}|\o^n&\leq&\int_X|(\pi^i-\pi^{i-1})\circ \hat F\circ (\pi^i-\pi^{i-1})|\o^n+2||\bar\pl\pi^i||^2_{L^2}\nonumber\\
&&+||\bar\pl\pi^{i-1}||^2_{L^2}+2||\bar\pl\pi^i||_{L^2}||\bar\pl\pi^{i-1}||_{L^2}.\nonumber
\eea
Thus, along a subsequence $A_j$ we have the following:
\bea
\int_X|\hat F_j^{Q^i}-\mu(Q^i)I|\o^n&\leq&\int_X|(\pi^i-\pi^{i-1})\circ(\hat F_j-\Psi_j)\circ (\pi^i-\pi^{i-1})|\o^n+2||\bar\pl_j\pi^i_j||^2_{L^2}\nonumber\\
&&+||\bar\pl_j\pi_j^{i-1}||^2_{L^2}+2||\bar\pl_j\pi^i_j||_{L^2}||\bar\pl_j\pi^{i-1}_j||_{L^2}.\nonumber
\eea
Now we apply $\eqref{sffb}$ to get the desired estimate:
\be
\int_X|\hat F_j^{Q^i}-\mu(Q^i)I|\o^n\leq 6\int_X|\hat F_j-\Psi_j|\o^n.\nonumber
\ee
This completes the lemma.
\end{proof}

We now turn to convergence of the Seshadri filtrations. Since each quotient $Q^i$ in the Harder-Narasimhan filtration is semi-stable, it admits a Seshadri filtration
\be
\label{SF1}
0\subset \ti S^1_i\subset \ti S^2_i\subset\cdots\subset \ti S^q_i=Q^i,
\ee
where $\mu(\ti S^k_i)=\mu(Q^i)$ for all $k$, and each quotient $\ti Q^k_i=\ti S^k_i/\ti S^{k-1}_i$ is torsion free and stable. Here, just as in Section $\ref{firstsec}$, the subscript $i$ on $\ti S^k_i$ denotes that we are working with the Seshadri filtration from the $i$-th quotient from the Harder-Narasimhan filtration. Let
\be
0\subset \ti \pi^1_i\subset \ti \pi^2_i\subset\cdots\subset \ti \pi^{q-1}_i\subset Q^i\nonumber
\ee
be the filtration of $L^2_1$ projections corresponding to $\eqref{SF1}$. We show for all $k$ that the sequence of projections $(\ti\pi^k_i)_j$ converges to a limiting projection $(\ti\pi^k_i)_\infty$ in $L^2_1$ along a subsequence. To do so we need to check that this sequence satisfies the assumptions $i)$ and $ii)$ from Proposition $\ref{limit}$. To begin, fix a compact subset $K\subset (X\backslash Z)$. On $K$ we note that the projections $\pi^i$ defining the Harder-Narasimhan filtration are smooth, and using equation  $\eqref{connectiondecomp}$ we can see how a sequence of connections $A_j$ along the Yang-Mills flow decomposes onto this filtration. Because this decomposition is orthogonal with respect to $H_0$, we have that the second fundamental form terms and the induced connections on each quotient converge smoothly as well. This shows assumption $i)$ holds for the sequence of induced connections $A^{Q^i}_j$ on each quotient. To see assumption $ii)$, we use the following modification of the Chern-Weil formula:
\be
\label{CW}
\mu(\ti S^k_i)=\mu(Q^i)+\dfrac{1}{rk(\ti S^k_i)}(\int_X{\rm Tr}((\hat F_j^{Q^i}-\mu (Q^i) I)\circ (\ti\pi^k_i)_j)\o^n-||\bar\pl_j(\pi^k_i)_j||_{L^2}^2).
\ee 
Because $\mu(\ti S^k_i)=\mu(Q^i)$, we have
\be
||\bar\pl_j(\ti\pi^k_i)_j||_{L^2}^2=\int_X{\rm Tr}((\hat F_j^{Q^i}-\mu (Q^i) I)\circ (\ti\pi^k_i)_j)\o^n,\nonumber
\ee
which goes to zero by Lemma $\ref{AHES}$. This verifies assumption $ii)$. Thus we can apply Proposition $\ref{limit}$ to $(\ti\pi^k_i)_j$,  and get that the Seshadri filtration converges to a limiting filtration:
\be
0\subset (\ti \pi^1_i)_\infty\subset (\ti \pi^2_i)_\infty\subset\cdots\subset (\ti \pi^{q-1}_i)_\infty\subset Q^i_\infty.\nonumber
\ee
Since the norms of the second fundamental forms go to zero, this filtration decomposes $Q^i_\infty$ into a direct sum of quotients $(\ti Q^k_i)_\infty$. In fact, at this point in the argument we can apply the author's previous work to construct an isomorphism between  $\bigoplus_k\ti Q^k_i$ and  $\bigoplus_k(\ti Q^k_i)_\infty$, which follows exactly from Theorem 1 from $\cite{J2}$ . Thus, on $X\backslash Z$ we have an isomorphism between $Q^i_\infty$ and $Gr^s(Q^i)$. In fact, as described explicitly in the proof of Theorem 1 from  $\cite{J2}$, the construction of an isomorphism $Q^i_\infty$ and $Gr^s(Q^i)$ starts by considering the subsheaf of lowest rank from $Gr^s(Q^i)$, and working with subsheaves of higher and higher rank until an isomorphism has been constructed for the entire filtration. Since this process is independent of $i$, applying this argument inductively to each quotient sheaf $Q^i$, we construct an isomorphism between $Gr^{hns}(E)$ and $\bigoplus_i\bigoplus_k (\ti Q^k_i)_\infty$. Since the direct sum of quotients from any Seshadri filtration is unique, we know $\bigoplus_i\bigoplus_k (\ti Q^k_i)_\infty$ is isomorphic to $\bigoplus_p\hat Q^p_\infty$ (from $\eqref{decomp2}$), proving Proposition $\ref{isom1}$.

We have now constructed, on $X\backslash Z$, the following isomorphism:
\be
\label{isom2}
Gr^{hns}(E)\cong E_\infty.
\ee
In order to prove Theorem $\ref{main theorem}$,  we need to show this isomorphism can be extended to an isomorphism between $Gr^{hns}(E)^{**}$ and the Bando-Siu extension $\hat E_\infty$ on all of $X$. As a first step we show that $E_\infty$ can be extended over $Z$ as the reflexive sheaf $Gr^{hns}(E)^{**}$. To do so, notice that:
\be
\label{finalcong}
\Gamma(X\backslash Z,Gr^{hns}(E))\cong\Gamma(X\backslash Z,Gr^{hns}(E)^{**}),
\ee
since $Gr^{hns}(E)$ is locally free on $X\backslash Z$. Since all holomorphic functions can be extended over $Z$ by a result of Shiffman from $\cite{Shiff}$, and because $Gr^{hns}(E)^{**}$ is reflexive it is defined by $Hom(Gr^s(E)^{*},\cO)$, we have: 
\be
\Gamma(X\backslash Z,Gr^{hns}(E)^{**})\cong\Gamma(X,Gr^{hns}(E)^{**}).\nonumber
\ee
Combining this isomorphism with $\eqref{finalcong}$ we have
\be
\Gamma(X\backslash Z,Gr^{hns}(E))\cong\Gamma(X,Gr^{hns}(E)^{**}).
\ee
Thus, using $\eqref{isom2}$, we see that $E_\infty$ extends over the singular set $Z$ as the reflexive sheaf $Gr^{hns}(E)^{**}$.

As stated in $\cite{BaS}$, the existence of a Bando-Siu extension $\hat E_\infty$ is a consequence of Bando's removable singularity theorem $\cite{Ba}$ and Siu's slicing theorem $\cite{Siu1}$. Aside from those references, we also direct the reader to $\cite{J2}$ for a detailed description of the Bando-Siu sheaf extension in this case. The relevant fact for us is the uniqueness of the sheaf extension which is proven in $\cite{Siu1}$.  This uniqueness theorem is characterized by the fact that given any other reflexive extension (in our case $Gr^{hns}(E)^{**}$), there exists a sheaf isomorphism $\phi:\hat E_\infty\longrightarrow Gr^{hns}(E)^{**}$ on $X$, which restricts to the isomorphism constructed in Proposition $\ref{isom1}$ on $X\backslash Z$. This completes the proof of Theorem $\ref{main theorem}$.

Thus even though we do not know whether $Z_{an}$ depends on the subsequence $A_j$, the limiting reflexive sheaf $\hat E_\infty$ (defined on all of $X$) is canonical and does not depend on the choice of subsequence. We have the following corollary of Theorem $\ref{main theorem}$:

\begin{corollary}
\label{cor}
The algebraic singular set $Z_{alg}$ is contained in the analytic singular set $Z_{an}$.
\end{corollary}

\begin{proof}

We prove $Z_{alg}\subseteq Z_{an}$. Suppose there exists a point $x_0\in Z_{alg}$ which is not in $Z_{an}$. We know there exists a quotient $\ti Q^k_i$ from $Gr^{hns}(E)$ such that $\ti Q^k_i$ is not locally free at $x_0$. Yet by Theorem $\ref{main theorem}$ we know $Q^i$ is isomorphic to some $Q^i_\infty$ from the direct sum $E_\infty=\oplus_p\hat Q^p_\infty$, and since $E_\infty$ is a vector bundle off $Z_{an}$ we know $\ti Q^k_i$ is locally free there. 

\end{proof}

\end{normalsize}


\newpage

\begin{thebibliography}{4}

{\small

\bibitem{AB} M.F. Atiyah and R. Bott,
{\it The Yang-Mills equations over Riemann surfaces}.
Phil. Trans. Roy. Soc. London A {\bf 308} (1983), 532-615.

\bibitem{Ba} S. Bando, 
{\it Removable singularities for holomorphic vector bundles}.
Tohoku Math. J. (2) {\bf 43} (1991), no. 1, 61-67.

\bibitem{BaS} S. Bando and Y.-T. Siu,
{\it Stable sheaves and Einstein-Hermitian metric},
Geometry and Analysis on Complex Manifolds, World Sci. Publ., River Edge, NJ (1994), 39-50.

\bibitem{DW} G. Daskalopoulos and R. Wentworth,
{\it Convergence properties of the Yang-Mills flow on K\"ahler surfaces}.
J. Reine Angew. Math. {\bf 575} (2004), 69-99.

\bibitem{DW2} G. Daskalopoulos and R. Wentworth,
{\it On the blow-up set of the Yang-Mills flow on K\"ahler surfaces}.
Math. Z. {\bf 256} (2007), no. 2, 301-310.

\bibitem{Don0} S.K. Donaldson, 
{\it A new proof of a theorem of Narasimhan and Seshadri}.
 J. Differential Geom. {\bf 18} (1983), no. 2, 269-277. 

\bibitem{Don1} S.K. Donaldson, 
{\it Anti self-dual Yang-Mills connections over complex algebraic surfaces and stable vector bundles}.
Proc. London Math. Soc. (3) {\bf 50} (1985), 1-26.

\bibitem{Don2} S.K. Donaldson, 
{\it Infinite determinants, stable bundles, and curvature}.
Duke Math. J. {\bf 54} (1987), 231-247.

\bibitem{Don3} S.K. Donaldson, 
{\it Lower bounds on the Calabi functional}.
J. Differential Geometry {\bf 70} (2005), 453-472.

\bibitem{DonK} S.K. Donaldson and P.B. Kronheimer,
{\it The geometry of four-manifolds}.
Oxford Mathematical Monographs. Oxford Science Publications. The Clarendon Press, Oxford University Press, New York  (1990).

\bibitem{GH} P. Griffiths and J. Harris,
{\it Principles of Algebraic Geometry}.
John Wiley $\&$ Sons, (1978).

\bibitem{HT} M.-C. Hong and G. Tian,  
{\it Asymptotical behaviour of the Yang-Mills flow and singular Yang-Mills connections}.
Math. Ann. {\bf 330}  (2004),  441-472.

\bibitem{J1} A. Jacob, {\it Existence of approximate Hermitian-Einstein structures on semi-stable bundles}. Asian J. Math. (to appear)

\bibitem{J2} A. Jacob, {\it The limit of the Yang-Mills flow on semi-stable bundles}. J. Reine Angew. Math. (to appear)

\bibitem{Kob} S. Kobayashi,
{\it Differential geometry of complex vector bundles}.
Publications of the Mathematical Society of Japan, {\bf 15}. Kan\^o Memorial Lectures, {\bf 5}. Princeton University Press, Princeton, NJ (1987).

\bibitem{NS} M.S. Narasimhan, and C.S. Seshadri, 
{\it Stable and unitary bundles on a compact Riemann surface}. Ann. of Math {\bf 82} (1965), 540-564.

\bibitem{PSSW} D.H. Phong, J. Song, J. Sturm and B. Weinkove,
{\it The K\"ahler-Ricci flow and the $\bar\pl$ operator on vector fields}.
J. Differential Geom.  {\bf 81}  (2009),  no. 3, 631-647.

\bibitem{Po} D. Popovivi,
{\it A simple proof of a theorem by Uhlenbeck and Yau}.
 Math. Z. {\bf 250} (2005), no. 4, 855-872. 

\bibitem{Shiff} B. Shiffman,
{\it On the removal of singularities of analytic sets}.
Michigan Math. J.  {\bf 15}  (1968),  111-120.

\bibitem{SB} B. Sibley and R. Wentworth,
{\it Analytic cycles, Bott-Chern forms, and singular sets for the Yang-Mills flow on Kaehler manifolds}. arXiv:1402.3808 

\bibitem{Simp} C. Simpson, 
{\it Constructing variations of Hodge structure using Yang-Mills theory and applications to uniformization}.
J. Amer. Math. Soc.  {\bf 1}  (1988),  no. 4, 867-918.

\bibitem{Siu1} Y.-T. Siu,
{\it A Hartogs type extension theorem for coherent analytic sheaves}.
Ann. of Math. (2) {\bf 93} (1971), no. 1, 166-188

\bibitem{Siu}Y.-T. Siu,
{\it Lectures on Hermitian-Einstein metrics for stable bundles and K\"ahler-Einstein metrics}.
Birk\"auser Verlag, Basel, (1987).

\bibitem{U} K. Uhlenbeck, 
{\it Connections with $L^p$ bounds on curvature}.
Comm. Math. Phys. {\bf 83} (1982), no. 1, 31-42.

\bibitem{UY} K. Uhlenbeck and S.-T. Yau, 
{\it On the existence of Hermitian-Yang-Mills connections in stable vector bundles}.
Comm. Pure and Appl. Math. {\bf 39-S} (1986), 257-293.


}
\end{thebibliography}
\end{document}